\documentclass[11pt]{article}

\usepackage[T1]{fontenc}
\usepackage{lmodern}
\usepackage{etoolbox}
 
\makeatletter
\patchcmd{\@maketitle}{\LARGE \@title}{\LARGE\bfseries\@title}{}{}
\makeatother

\usepackage{amsmath, amssymb, amsthm} 
\usepackage[margin=2.27cm]{geometry}
\usepackage{color, graphicx, soul}
\usepackage{caption}
\usepackage{appendix}
\usepackage{titlesec}
\titlelabel{\thetitle.\quad}

\usepackage{hyperref}
\hypersetup{
	colorlinks=true,        % false: boxed links; true: colored links
	linkcolor=blue,         % color of internal links
	citecolor=blue,         % color of links to bibliography
	urlcolor=blue           % color of external links
}

\usepackage[overload]{empheq}
\definecolor{myblue}{rgb}{.9, .9, 1}

\makeatletter
\def\th@plain{%
	\thm@notefont{}% same as heading font
	\itshape % body font
}
\def\th@definition{%
	\thm@notefont{}% same as heading font
	\normalfont % body font
}

\renewenvironment{proof}[1][\proofname]{\par
	\normalfont
	\topsep0\p@\@plus3\p@ \trivlist
	\item[\hskip\labelsep\itshape
	#1\@addpunct{.}]\ignorespaces
}{%
	\qed\endtrivlist
}
\makeatother

\newtheorem{theorem}{Theorem}[section]
\newtheorem{lemma}[theorem]{Lemma}
\newtheorem{corollary}[theorem]{Corollary}
\newtheorem{proposition}[theorem]{Proposition}
\newtheorem{fact}[theorem]{Fact}

\theoremstyle{definition}
\newtheorem{definition}[theorem]{Definition}
\theoremstyle{definition}

\theoremstyle{definition}
\newtheorem{remark}[theorem]{Remark}

\usepackage[shortlabels]{enumitem}
\setlist[enumerate]{nosep}

\allowdisplaybreaks

\newcommand{\menge}[2]{\{{#1}~\big |~{#2}\}}

\newcommand{\ball}[2]{{\ensuremath{\it I\hspace{-5pt}B}}(#1;#2)}
\newcommand{\bball}[2]{{\ensuremath{\it I\hspace{-5pt}B}}\big(#1;#2\big)}
\newcommand{\uball}{{\ensuremath{\it I\hspace{-5pt}B}}}
\newcommand{\scal}[2]{\left\langle {#1},{#2} \right\rangle}

\newcommand{\pnX}[1]{N^{\rm prox}_{#1}} %proximal with B=X

\newcommand{\To}{\ensuremath{\rightrightarrows}}

\newcommand{\NN}{\ensuremath{\mathbb N}}
\newcommand{\nnn}{\ensuremath{{n\in{\mathbb N}}}}

\newcommand{\RR}{\ensuremath{\mathbb R}}
\newcommand{\RP}{\ensuremath{\mathbb{R}_+}}
\newcommand{\RPP}{\ensuremath{\mathbb{R}_{++}}}
\newcommand{\RM}{\ensuremath{\mathbb{R}_-}}

\newcommand{\argmin}{\ensuremath{\operatorname*{argmin}}}

\newcommand{\inte}{\ensuremath{\operatorname{int}}}
\newcommand{\reli}{\ensuremath{\operatorname{ri}}}

\newcommand{\aff}{\ensuremath{\operatorname{aff}}}

\newcommand{\epi}{\ensuremath{\operatorname{epi}}}

\newcommand{\Fix}{\ensuremath{\operatorname{Fix}}}
\newcommand{\Id}{\ensuremath{\operatorname{Id}}}

\def\ox{\overline{x}}

\def\hbeta{{\widehat\beta}}

\definecolor{darkred}{rgb}{.9, .0, .0}

\begin{document}

\title{Linear convergence of the generalized Douglas--Rachford algorithm for feasibility problems}

\author{
Minh N.\ Dao\thanks{CARMA, University of Newcastle, Callaghan, NSW 2308, Australia. 
E-mail: \texttt{daonminh@gmail.com}}
~and~
Hung M.\ Phan\thanks{Department of Mathematical Sciences, Kennedy College of Sciences, University of Massachusetts Lowell, MA 01854, USA.
E-mail: \texttt{hung\char`_phan@uml.edu}.}}

\date{April 16, 2018}

\maketitle

% REQUIRED
\begin{abstract}
In this paper, we study the generalized Douglas--Rachford algorithm and its cyclic variants which include many projection-type methods such as the classical Douglas--Rachford algorithm and the alternating projection algorithm. Specifically, we establish several local linear convergence results for the algorithm in solving feasibility problems with finitely many closed possibly nonconvex sets under different assumptions. Our findings not only relax some regularity conditions but also improve linear convergence rates in the literature. In the presence of convexity, the linear convergence is global.
\end{abstract}

{\small
\noindent{\bfseries AMS Subject Classifications:}
{Primary: 
47H10, % Fixed-point theorems
49M27; % decomposition methods
Secondary: 
% 47H09, % Nonexpansive mappings, and their generalizations
41A25, % Rate of convergence, degree of approximation
% 49J52, % Nonsmooth analysis
% 49M37, % Methods of nonlinear programming type
65K05, % Mathematical programming
65K10, % optimization and variational techniques
90C26. % Nonconvex programming, global optimization
}

}

\maketitle

\section{Introduction}

The \emph{feasibility problem} of finding a point in the intersection of closed constraint sets is of central importance in diverse areas of mathematics and engineering. Many methods have been proposed for this problem, and most of them naturally involve \emph{nearest point projectors} and their variants with respect to underlying sets. We refer the readers to \cite{BB96} and the references therein for more reviews and discussions on projection-type methods and applications.

Among methods for feasibility problems, the \emph{Douglas--Rachford (DR) algorithm} has recently drawn much attention due to its interesting features which mysteriously allow for its success in both convex and nonconvex settings. The DR algorithm was first formulated in \cite{DR56} for solving nonlinear heat flow problems numerically. Since then, it has been emerged in optimization theory and applications thanks to the seminal work \cite{LM79}; more specifically, the DR algorithm was extended to the problem of finding a zero of the sum of two maximally monotone operators. When specialized to normal cone operators, the DR algorithm can be used for solving feasibility problems.

In the convex case, the convergence theory of the DR algorithm is well developed. In particular, weak convergence of the DR sequence to a fixed point was proved in \cite{LM79} while that of the shadow sequence to a solution was proved in \cite{Sva11}. When the problem is \emph{infeasible}, it was shown in \cite{BCL04} that the shadow sequence is bounded with cluster points solving a best approximation problem, in \cite{BDM16} that the entire shadow sequence is weakly convergent if one set is an affine subspace, and in \cite{BM17} that the affine subspace condition can be removed. For nonconvex problems, the DR algorithm enjoys successful applications including, notably, phase retrieval, protein folding, and Sudoku, see, e.g., \cite{ERT07} (in which the DR algorithm is referred to as a special case of the \emph{difference map}) and the references therein, even though the supporting theory is far from being complete.

Regarding convergence rate analysis of the DR algorithm, various results have been obtained, some of which even apply to certain nonconvex settings. For example, \cite{BBNPW14} proved linear convergence for two subspaces, \cite{BD17, BDNP16b} proved several finite convergence results, in particular, when the sets involved are subspaces, halfspaces, epigraphs of convex functions, and finite sets, while \cite{DT17, HL13, Pha14} proved the linear convergence rate under some regularity conditions, see also \cite{BNP15}. However, it was observed in \cite{BDNP16a} that without regularity assumptions, linear convergence of the DR algorithm may fail even for simple cases in the Euclidean plane. Recently, linear convergence was studied in \cite{DP16} for the so-called \emph{generalized Douglas--Rachford algorithm}, which is a generalization of several projection-type methods including the DR algorithm and the \emph{alternating projection algorithm}. Although the generalized DR algorithm in fact appeared earlier in \cite[Example~3.5]{BST15}, little is known about its rate of convergence.

The goal of this paper is to provide linear convergence results for the generalized DR algorithm and its \emph{cyclic} variants. To obtain linear convergence, the essential assumptions employed in recent works are \emph{superregularity} of sets, \emph{linear regularity}, \emph{strong regularity}, and \emph{affine-hull regularity} of the system of constraint sets.
Indeed, these concepts were previously utilized in \cite{BLPW13a,BLPW13b,BPW13,DIL15,HL13,LLM09,LTT16,NR16} to prove linear convergence for the alternating projection algorithm. In the literature, linear regularity and strong regularly are also known as \emph{subtransversality} and \emph{transversality}, respectively.
Our main contributions include various improvements in this direction. Precisely, in one of the new results, we obtain \emph{better} linear convergence rate under the \emph{same} assumptions used in recent studies (see Theorem~\ref{t:cDR} and Remark~\ref{r:better_rate}). In another new result, we \emph{lessen} the assumptions required for linear convergence by exploiting the flexibility of generalized DR operators (see Theorem~\ref{t:super} and Corollary~\ref{c:commonpoint}). To the best of our knowledge, the latter is the \emph{first} result that puts together linear regularity and operator parameters to obtain linear convergence for variants of the DR algorithm. In the convex case, we show that the linear convergence is global (see Theorem~\ref{t:cvxcDR} and Corollary~\ref{c:gDR}). On the one hand, the linear convergence results in this work can be seen as a continuation of \cite{DP16}. On the other hand, it is also worth noting that the generalized DR algorithm involving at most one reflection is convergent in convex case even when the problem is infeasible (see Theorem~\ref{t:cvg}). Moreover, we present several new features of the generalized DR algorithm, for example, the parameters may decide whether the iterations converge to the set of fixed points or to the intersection. Our results provide another step toward understanding the behavior of the celebrated DR algorithm.

The remainder of the paper is organized as follows. Section~\ref{s:pre} contains preliminary materials needed for our analysis. While part of Section~\ref{s:pre} was previously presented in \cite{DP16}, we include it here for completeness. In Sections~\ref{s:qfFm} and \ref{s:qcoer}, we provide some improvements and new results on \emph{quasi firm Fej\'er monotonicity} and \emph{quasi coercivity} which are key ingredients in convergence rate analysis. Section~\ref{s:convg} presents our main results on linear convergence for the generalized DR algorithm. Finally, the conclusion is given in Section~\ref{s:con}.

\section{Auxiliary results}
\label{s:pre}
Unless otherwise stated, $X$ is a Euclidean space with inner product $\scal{\cdot}{\cdot}$ and induced norm $\|\cdot\|$. The nonnegative integers are $\NN$, the real numbers are $\RR$, while $\RP := \menge{x \in \RR}{x \geq 0}$ and $\RPP := \menge{x \in \RR}{x >0}$.
If $w \in X$ and $\rho \in \RP$, then $\ball{w}{\rho} :=\menge{x \in X}{\|x -w\| \leq \rho}$ is the closed ball centered at $w$ with radius $\rho$, and $\uball$ simply stands for the unit ball $\ball{0}{1}$. The set of fixed points of a set-valued operator $T:X\rightrightarrows X$ is $\Fix T:=\menge{x\in X}{x\in Tx}$. Also, $\aff A$ represents the affine hull of a set $A$ and $L^\perp:=\menge{u\in X}{\forall x\in L,\ \scal{u}{x}=0}$ is the orthogonal complementary subspace of $L$.

\subsection{Distance function and relaxed projectors}

Let $C$ be a nonempty subset of $X$. The \emph{distance function} to $C$ is defined by
\begin{equation}
d_C\colon X\to\RR\colon x\mapsto \inf_{c \in C}\|x -c\|
\end{equation}
and the \emph{projector} onto $C$ is
\begin{equation}
P_C\colon X\To C\colon x\mapsto \argmin_{c \in C}\|x -c\| =\menge{c \in C}{\|x -c\| =d_C(x)}.
\end{equation}
The \emph{relaxed projector} for $C$ with parameter $\lambda\in \RP$ is defined by
\begin{equation}
P_C^\lambda :=(1- \lambda)\Id +\lambda P_C.
\end{equation}
Note that $P_C^0 =\Id$ is the \emph{identity} operator, $P_C^1 =P_C$ is the projector onto $C$, and $P_C^2 =R_C :=2P_C -\Id$ is the so-called \emph{reflector} across $C$. 

We now collect some useful properties of relaxed projectors.

\begin{lemma}
\label{l:projL}
Let $L$ be an affine subspace of $X$. Then the following hold:
\begin{enumerate}
\item\label{l:projL_perp} 
$P_L$ is an affine operator, and for every $x\in X$, $x -P_Lx\in (L- L)^\perp$. 
\item\label{l:projL_dist} 
$\forall x\in X$, $\forall z\in X$,\quad $\|x -z\|^2 =\|P_Lx -P_Lz\|^2 +\|(x -P_Lx) -(z -P_Lz)\|^2$. 
\item\label{l:projL_ball}
$\forall x\in X$, $\forall w\in L$,\quad $\|x -w\|^2 =\|P_Lx -w\|^2 +\|x -P_Lx\|^2$. Consequently, $\|P_Lx -w\|\leq \|x -w\|$.
\end{enumerate}
\end{lemma}
\begin{proof}
\ref{l:projL_perp}: This follows from \cite[Corollary~3.22]{BC17}.

\ref{l:projL_dist}: Let $x\in X$ and let $z\in X$. Then \ref{l:projL_perp} implies that $\scal{P_Lx -P_Lz}{x -P_Lx} =0$ and $\scal{P_Lx -P_Lz}{z -P_Lz} =0$, which yields
\begin{subequations}
\begin{align}
\|x -z\|^2 &=\|(P_Lx -P_Lz) +[(x -P_Lx) -(z -P_Lz)]\|^2 \\ 
&=\|P_Lx -P_Lz\|^2 +\|(x -P_Lx) -(z -P_Lz)\|^2.	
\end{align}		
\end{subequations}

\ref{l:projL_ball}: Apply \ref{l:projL_dist} to $z =w\in L$. 
\end{proof}

\begin{lemma}
\label{l:aff}
Let $C$ be a nonempty closed subset of $X$, let $L$ be an affine subspace of $X$ containing $C$, and let $\lambda\in \RR$. Then the following hold:
\begin{enumerate}
\item\label{l:aff_imag}
$P_C^\lambda(L)\subseteq L$.
\item\label{l:aff_comm} 
$P_CP_L =P_C =P_LP_C$ and $P_LP_C^\lambda =P_C^\lambda P_L$.
\item\label{l:aff_diff}
$(\Id -P_L)P_C^\lambda =(1 -\lambda)(\Id -P_L)$.
\item\label{l:aff_dist}
$\forall x\in X$,\quad
$d_C^2(P_L^\lambda x)=d_C^2(P_Lx)+(1-\lambda)^2d_L^2(x)$. In particular, $d_C^2(x) =d_C^2(P_Lx) +d_L^2(x)$.
\item\label{l:aff_proj}
If $Q$ is a nonempty closed subset of $X$ such that $Q\subseteq (L- L)^\perp$, then
\begin{subequations}
\begin{align}
\forall y\in L,\ \forall q\in Q,\quad
&P_{C+Q}(y+q)=q+P_{C+Q}y =q+P_{C}y,\\
&d_{C+Q}(y+q)=d_{C+Q}(y)=d_C(y).
\end{align}
\end{subequations}
\end{enumerate}
\end{lemma}
\begin{proof}
\ref{l:aff_imag}: Let $x\in L$. Since $P_Cx\subseteq C\subseteq L$ and $L$ is an affine subspace, we have that $P_C^\lambda x= (1-\lambda)x+ \lambda P_Cx\subseteq (1-\lambda)L+ L\subseteq L$.

\ref{l:aff_comm}: It follows from Fact~\ref{l:projL}\ref{l:projL_perp} that $P_L$ is an affine operator, and from \cite[Lemma~3.3]{BLPW13a} that $P_CP_L =P_C =P_LP_C$. Thus, 
\begin{equation}
P_LP_C^\lambda= P_L\big((1-\lambda)\Id+ \lambda P_C\big)= (1-\lambda)P_L+ \lambda P_LP_C
= (1-\lambda)P_L+ \lambda P_CP_L= P_C^\lambda P_L.
\end{equation}

\ref{l:aff_diff}: Using \ref{l:aff_comm}, we obtain that
\begin{subequations}
\begin{align}
(\Id-P_L)P_C^\lambda&= P_C^\lambda- P_L P_C^\lambda= P_C^\lambda- P_C^\lambda P_L \\
&= (1-\lambda)\Id +\lambda P_C- (1-\lambda)P_L- \lambda P_CP_L= (1 -\lambda)(\Id -P_L).
\end{align}
\end{subequations}

\ref{l:aff_dist}: In Lemma~\ref{l:projL}\ref{l:projL_ball}, substituting $x$ by $P_L^\lambda x$ and using \ref{l:aff_diff}, we get
\begin{equation}
\|P_L^\lambda x-w\|^2=\|P_Lx-w\|^2+\|(\Id-P_C)P_L^\lambda x\|^2
=\|P_Lx-w\|^2+(1-\lambda)^2\|x-P_L x\|^2.
\end{equation}
Taking the infimum over $w\in C$ yields
\begin{equation}\label{e:171015a}
d_C^2(P_L^\lambda x)=d_C^2(P_L^\lambda x)+(1-\lambda)^2\|x-P_Lx\|^2.
\end{equation}
For the second conclusion, we apply \eqref{e:171015a} with $\lambda=0$.

\ref{l:aff_proj}: Let $y\in L$ and let $q\in Q$. It is straightforward to see that $P_{C+Q}(y+q) =P_{C+Q+q}(y+q) =q +P_{C+Q}y$. Now for all $c\in C$, since $c- y\in L- L$, we have
\begin{subequations}
\begin{align}
\|(c+q)-y\|^2&= \|c-y\|^2+ 2\scal{q}{c-y}+ \|q\|^2= \|c-y\|^2+ \|q\|^2 \\
&\geq \|c-y\|^2\geq d_C^2(y)\geq d_{C+Q}^2(y).
\end{align}
\end{subequations}
So $c+q\in P_{C+Q}y$ if and only if $q= 0$ and $c\in P_Cy$. Therefore, $P_{C+Q}y= P_Cy$, which completes the proof. 
\end{proof}

\subsection{Normal cones and regularity of sets}

Let $C$ be a nonempty subset of $X$ and let $x\in C$. The \emph{proximal normal cone} to $C$ at $x$ (see \cite[Section~2.5.2, D]{Mor06} and \cite[Example~6.16]{RW98}) is defined by
\begin{equation}
\pnX{C}(x) :=\menge{\lambda(z-x)}{z\in P_C^{-1}(x),\ \lambda\in\RP},
\end{equation}
and the \emph{limiting normal cone} to $C$ at $x$ (see \cite[Definition~1.1(ii) and Theorem~1.6]{Mor06}) can be given by
\begin{equation}\label{e:170624b}
N_C(x) :=\menge{u \in X}{\exists x_n\to x, u_n\to u \text{~with~} x_n\in C, u_n\in \pnX{C}(x_n)}.
\end{equation}

Let $w\in X$, $\varepsilon \in \RP$, and $\delta \in \RPP$. We recall from \cite[Definition~8.1]{BLPW13a} and \cite[Definition~2.9]{HL13} that $C$ is \emph{$(\varepsilon,\delta)$-regular} at $w$ if
\begin{equation}
\left.\begin{aligned}
&x,y\in C\cap\ball{w}{\delta},\\
&u\in\pnX{C}(x)
\end{aligned}\right\} \implies
\scal{u}{x-y}\geq -\varepsilon\|u\|\cdot\|x-y\|
\end{equation}
and \emph{$(\varepsilon, \infty)$-regular} at $w$ if it is $(\varepsilon, \delta)$-regular for all $\delta\in \RPP$. The set $C$ is \emph{superregular} at $w$ if for all $\varepsilon \in \RPP$, there exists $\delta \in \RPP$ such that $C$ is $(\varepsilon,\delta)$-regular at $w$.

\begin{fact}
\label{f:ImBall}
Let $C$ be a nonempty closed subset of $X$, $w\in C$, $\lambda\in \RP$, and $\delta\in \RPP$. Then the following hold:
\begin{enumerate}
\item\label{f:ImBall_P}
$P_C^\lambda(\ball{w}{\delta/2})\subseteq \ball{w}{(1+\lambda)\delta/2}$. In particular, $P_C(\ball{w}{\delta/2})\subseteq C\cap \ball{w}{\delta}$.
\item\label{f:ImBall_reg} 
If $\lambda\in \left]0, 2\right]$ and $C$ is $(\varepsilon, \delta)$-regular at $w$ with $\varepsilon \in \left[0, 1/3\right]$, then $P_C^\lambda(\ball{w}{\delta/2})\subseteq \ball{w}{\delta/\sqrt{2}}$.
\end{enumerate}
\end{fact}
\begin{proof}
This follows from \cite[Lemma~3.4(i) and Proposition~3.5]{DP16}.
\end{proof}

\subsection{Regularity of set systems}

In this section, $m$ is a positive integer, $I:= \{1, \dots, m\}$, and $\{C_i\}_{\in I}$ is a system of closed subsets of $X$.
Recall that $\{C_i\}_{\in I}$ is \emph{linearly regular} with modulus $\kappa \in\RPP$ (or \emph{$\kappa$-linearly regular}) on a subset $U$ of $X$ if 
\begin{equation}
\label{e:linreg}
\forall x\in U,\quad d_C(x) \leq \kappa\max_{i\in I}d_{C_i}(x),
\quad\text{where}\quad C :=\bigcap_{i\in I}C_i.
\end{equation}
We say that $\{C_i\}_{i\in I}$ is \emph{linearly regular} around $w \in X$ if there exist $\delta \in \RPP$ and $\kappa\in \RPP$ such that $\{C_i\}_{i\in I}$ is $\kappa$-linearly regular on $\ball{w}{\delta}$. The system $\{C_i\}_{\in I}$ is said to be \emph{boundedly linearly regular} if for every bounded set $S$ of $X$, there exists $\kappa_S\in \RPP$ such that $\{C_i\}_{i\in I}$ is $\kappa$-linearly regular on $S$.
Interested readers can find more discussion on linear regularity in \cite{BB96, BBL99, BNP15, DP16}.

The following is a generalization of strong regularity, see, e.g., \cite[Definition~2.3]{DP16}, and affine-hull regularity \cite[Definition~2.1]{Pha14}. 

\begin{definition}[$L$-regularity of set systems]
\label{d:strgreg}
Let $w\in \bigcap_{i\in I} C_i$ and let $L$ be an affine subspace of $X$ that contains $w$. The system $\{C_i\}_{\in I}$ is said to be \emph{$L$-regular} at $w$ if
\begin{equation}\label{e:strgreg}
\sum_{i\in I} u_i =0 \text{ \ and \ } u_i\in N_{C_i}(w)\cap(L-w)
\implies \forall i\in I,\ u_i=0.
\end{equation}
We simply say that $\{C_i\}_{i\in I}$ is \emph{strongly regular} at $w$ when $L=X$, and say that $\{C_i\}_{i\in I}$ is \emph{affine-hull regular} at $w$ when $L=\aff\bigcup_{i\in I}C_i$. In the case $I=\{1,2\}$, condition \eqref{e:strgreg} can be rewritten as
$N_{C_1}(w)\cap (-N_{C_2}(w))\cap(L-w)=\{0\}$.
\end{definition}

By definition, if the system $\{C_i\}_{i\in I}$ is $L$-regular at $w$, then so are all the subsystems. As shown in \cite{DP16}, strong regularity in Definition~\ref{d:strgreg} is equivalent to \cite[Definition~1(vi)]{Kru06} and \cite[Definition~3.2]{HL13}.

In what follows, $|I|$ denotes the number of elements in the set $I$.
\begin{proposition}[$L$-regularity implies linear regularity]
\label{p:Lstr}
Let $w\in C:=\bigcap_{i\in I} C_i$ and let $L$ be an affine subspace of $X$ containing $\bigcup_{i\in I}C_i$. Suppose that $\{C_i\}_{\in I}$ is $L$-regular at $w$. Then the following hold:
\begin{enumerate}
\item 
\label{p:Lstr_lin}
$\{C_i\}_{i\in I}$ is linearly regular around $w$.
\item 
\label{p:Lstr_aff}
If $m =|I|\geq 2$, then $L =\aff\bigcup_{i\in J}C_i$ whenever $J\subseteq I$ and $|J|\geq 2$. 
\item
\label{p:Lstr_sub}
Every subsystem of $\{C_i\}_{i\in I}$, including itself, is affine-hull regular at $w$.
\end{enumerate}
Consequently, if $\{C_i\}_{i\in I}$ is strongly regular at $w$, then $X=\aff\bigcup_{i\in I}C_i$.
\end{proposition}
\begin{proof}
\ref{p:Lstr_lin}: Consider the system $\{C_i\}_{i\in I}$ within $L$, then it is strongly regular at $w$ within $L$. So we learn from \cite[Theorem~1(ii)]{Kru06} that $\{C_i\}_{i\in I}$ is linearly regular around $w$ within $L$, i.e., there exist $\delta,\kappa\in \RPP$ such that
\begin{equation}\label{e:170731a}
\forall y\in \ball{w}{\delta}\cap L,\quad
d_{C}(y)\leq\kappa \max_{i\in I} d_{C_i}(y).
\end{equation}
Let $x\in \ball{w}{\delta}$ and $y =P_Lx\in L$. By Lemma~\ref{l:projL}\ref{l:projL_ball}, $y\in \ball{w}{\delta}\cap L$. Now by Lemma~\ref{l:aff}\ref{l:aff_dist}, $d_{C}^2(x) =d_{C}^2(y) + d_L^2(x)$ and $d_{C_i}^2(x) =d_{C_i}^2(y) +d_L^2(x)$ for every $i\in I$. Combining with \eqref{e:170731a}, we obtain
\begin{equation}
d_{C}^2(x)\leq\kappa^2\max_{i\in I}d_{C_i}^2(y) +d_L^2(x)
\leq\max\{\kappa^2,1\}\max_{i\in I}d_{C_i}^2(x),
\end{equation}
and so $\{C_i\}_{i\in I}$ is $\max\{\kappa,1\}$-linearly regular at $w$ on $\ball{w}{\delta}$.

\ref{p:Lstr_aff}: Take two distinct indices $i,j\in J$. By assumption, $\{C_i, C_j\}$ is $L$-regular at $w$. Suppose that $L$ strictly contains $\aff(C_i\cup C_j)$. Then $\aff(C_i\cup C_j) -w$ is a proper subspace of $L-w$. So there exists a unit vector $u\in L-w$ such that $u$ is perpendicular to $\aff(C_i\cup C_j) -w$. Now define $y_i =w+u$ and $y_j =w-u$. On the one hand, by \cite[Lemma~3.2]{BLPW13a}, $P_{C_i} y_i =P_{C_i} w =w$ and $P_{C_j}y_j =P_{C_j} w =w$, which yield $u\in\pnX{C_i}(w)\cap(L-w)$ and $-u\in\pnX{C_j}(w)\cap(L-w)$, respectively. On the other hand, $u +(-u) =0$ while $u\neq 0$. This contradicts the $L$-regularity of $\{C_i, C_j\}$. We must therefore have $L =\aff(C_i\cup C_j)$. Since $\aff(C_i\cup C_j)\subseteq \aff\bigcup_{i\in J} C_i\subseteq L$, it follows that $L =\aff\bigcup_{i\in J}C_i$.

\ref{p:Lstr_sub}: Let $\{C_i\}_{i\in J}$ be a subsystem of $\{C_i\}_{i\in I}$. If $|J|=1$, then $\{C_i\}_{i\in J}$ is automatically affine-hull regular at $w$. If $|J|\geq 2$, then from \ref{p:Lstr_aff}, we have $L=\aff\bigcup_{i\in J}C_i$, which implies affine-hull regularity of $\{C_i\}_{i\in J}$ at $w$.
\end{proof}

\begin{remark}[linear regularity does not imply affine-hull regularity]\label{r:lin-aff}
Proposition~\ref{p:Lstr} has shown that affine-hull regularity implies linear regularity. However, it is known that the reverse is not true, for example, in $\RR^2$, the system $\{\RP^2,\RM^2\}$ is linear regular, but not affine-hull regular at~$(0,0)$.
\end{remark}

\begin{remark}[affine-hull regularity of subsystems]
\label{r:str-lin}
Affine-hull regularity of every \emph{proper} subsystem $\{C_i\}_{i\in J}$ with $J\subsetneqq I$ and linear regularity of the entire system $\{C_i\}_{i\in I}$ do not imply affine-hull regularity of $\{C_i\}_{i\in I}$.
For example, in $\RR^2$, consider $C_1 =\menge{(\xi, \zeta)}{\xi +\zeta\leq 0}$, $C_2 =\menge{(\xi, \zeta)}{\xi -\zeta\leq 0}$, $C_3 =\menge{(\xi, \zeta)}{\xi\geq 0}$, and $w =(0,0)\in C_1\cap C_2\cap C_3$. Then one can check that $\{C_i\}_{i\in J}$ with $J\subsetneqq\{1,2,3\}$ is affine-hull regular at $w$, and that $\{C_1,C_2,C_3\}$ is linearly regular around $w$, but $\{C_1,C_2,C_3\}$ is not affine-hull regular at $w$. 
\end{remark}

Let $A$ and $B$ be nonempty subsets of $X$ and let $L$ be an affine subspace of $X$ containing $A\cup B$. We recall from \cite[Definiton~6.1]{BLPW13a} that the \emph{CQ-number} at a point $w\in X$ associated with $(A,B,L)$ and $\delta\in \RPP$ is defined by
\begin{equation}
\begin{aligned}
\theta_{A,B,L}(w,\delta) :=\sup\big\{\scal{u}{v}\,\big|\, 
&u\in \pnX{A}(a)\cap (L-a)\cap \uball,\ a\in A\cap\ball{w}{\delta},\\
&v\in -\pnX{B}(b)\cap (L-b)\cap \uball,\ b\in B\cap \ball{w}{\delta}\big\}
\end{aligned}
\end{equation}
and that the \emph{limiting CQ-number} at $w$ associated with $(A,B,L)$ is defined by
\begin{equation}
\overline\theta_{A,B,L}(w):=\lim\limits_{\delta\downarrow 0} \theta_{A,B,L}(w,\delta).
\end{equation}
Clearly, $\theta_{A,B,L}(w,\delta)=\theta_{B,A,L}(w,\delta)\leq 1$. When $L=X$, we simply write $(A,B)$ for $(A,B,L)$.

We end this section with a connection between the CQ-number and $L$-regularity for two sets. 
\begin{proposition}[$L$-regularity for two sets]
\label{p:str&CQ}
Let $A$ and $B$ be two nonempty subsets of $X$, let $w\in A\cap B$, and let $L$ be an affine subspace of $X$ containing $A\cup B$. Then the following are equivalent:
\begin{enumerate}
\item 
\label{p:str&CQ_str}
$\{A, B\}$ is $L$-regular at $w$.
\item 
\label{p:str&CQ_CQ}
The CQ-number $\theta_{A,B,L}(w,\delta) <1$ for some $\delta\in \RPP$.
\item
\label{p:str&CQ_lim} 
The limiting CQ-number $\overline\theta_{A,B,L}(w) <1$.	
\end{enumerate}	
\end{proposition}
\begin{proof}
By \cite[Example~7.2]{BLPW13a}, it suffices to show the equivalence of \ref{p:str&CQ_CQ} and \ref{p:str&CQ_lim}. In fact, if \ref{p:str&CQ_CQ} holds, then by the definition of the CQ-number, $\theta_{A,B,L}(w,\delta')\leq \theta_{A,B,L}(w,\delta) <1$
for all $\delta'\in \left]0, \delta\right]$,
which implies that $\overline\theta_{A,B,L}(w) =\lim_{\delta\downarrow 0} \theta_{A,B,L}(w,\delta) <1$. The converse is obvious.
\end{proof}

\subsection{Generalized Douglas--Rachford operator}

Let $A$ and $B$ be nonempty closed subsets of $X$, let $\lambda,\mu\in\left]0,2\right]$, and let $\alpha\in\RPP$.
The \emph{generalized Douglas--Rachford (gDR) operator} for the pair $(A, B)$ with parameters $(\lambda,\mu,\alpha)$, which was also considered in \cite{DP16}, is defined by
\begin{equation}\label{e:170729a}
T_{\lambda, \mu}^\alpha :=(1 -\alpha)\Id +\alpha P_B^\mu P_A^\lambda.
\end{equation}  
It is worth mentioning that $T_{1,1}^1= P_BP_A$ is the \emph{classical alternating projection operator} \cite{Bre65}, that $T_{2,2}^{1/2}= \frac{1}{2}(\Id+ R_BR_A)$ is the \emph{classical DR operator} \cite{DR56,LM79}, and that 
\begin{equation}
T_{2,2\alpha}^{1/2}= (1- \alpha)P_A+ \tfrac{\alpha}{2}(\Id+R_BR_A)
\end{equation}
is the \emph{relaxed averaged alternating reflection operator} \cite{Luk08}. In addition, if $B$ is an affine subspace of $X$, then by Lemma~\ref{l:projL}\ref{l:projL_perp}, $P_B$ is an affine operator, and therefore 
\begin{equation}
T_{1+\alpha,1+\alpha}^{1/(1+\alpha)}= (1- \alpha)P_BP_A+ \tfrac{\alpha}{2}(\Id+R_BR_A)
\end{equation}
is an affine combination of the classical alternating projection and DR operators.  

It is interesting to see that the shadow of any \emph{gDR step} on certain affine subspaces is again a gDR step. This phenomenon is referred to as \emph{affine reduction} in \cite[Section~3]{Pha14}.

\begin{lemma}[shadows of gDR steps]
\label{l:shadow}
Let $L$ be an affine subspace of $X$ containing $A\cup B$, $x\in X$, and $x_+\in T_{\lambda, \mu}^\alpha x$. Define $y =P_Lx$, $y_+ =P_Lx_+$, and set $\eta :=1 -\alpha +\alpha(1 -\lambda)(1 -\mu)$. Then the following hold:
\begin{enumerate}
\item
\label{l:shadow_step} 
$y_+\in T_{\lambda, \mu}^\alpha y$.
\item
\label{l:shadow_diff} 
$x_+ -y_+ =\eta(x -y)$. Consequently, $d_L(x_+) =|\eta|d_L(x)$.
\item 
\label{l:shadow_dist}
$\|x -x_+\|^2 =\|y -y_+\|^2 +(1 -\eta)^2\|x -y\|^2$.
\end{enumerate}
\end{lemma}
\begin{proof}
Since $x_+\in T_{\lambda, \mu}^\alpha x$, there exist $r \in P_A^\lambda x$ and $s\in P_B^\mu r$ such that $x_+ =(1 -\alpha)x +\alpha s$.
 
\ref{l:shadow_step}: Using Lemma~\ref{l:aff}\ref{l:aff_comm}, we have $P_Ls\in P_LP_B^\mu P_A^\lambda x =P_B^\mu P_A^\lambda P_Lx =P_B^\mu P_A^\lambda y$. By Lemma~\ref{l:projL}\ref{l:projL_perp}, $P_L$ is an affine operator, and hence
\begin{equation}
\label{e:y+}
y_+ =P_Lx_+
=(1 -\alpha)P_Lx +\alpha P_Ls 
\in (1-\alpha_n)y +\alpha P_B^\mu P_A^\lambda y =T_{\lambda, \mu}^\alpha y.
\end{equation}

\ref{l:shadow_diff}: It follows from Lemma~\ref{l:aff}\ref{l:aff_diff} that 
\begin{equation}\label{e:shdwDR1}
s -P_Ls =(1 -\mu)(r -P_Lr)
=(1 -\mu)(1 -\lambda)(x -P_Lx).
\end{equation} 
Combining with \eqref{e:y+} yields
\begin{equation}
x_+ -y_+
=(1 -\alpha)(x -P_Lx) +\alpha(s -P_Ls)
=\eta(x -P_Lx) =\eta(x -y).
\end{equation}

\ref{l:shadow_dist}: Apply Lemma~\ref{l:projL}\ref{l:projL_dist} to $z =x_+$ and take \ref{l:shadow_diff} into account. 
\end{proof}

Next, we study the fixed points of gDR operators.

\begin{lemma}[fixed points of gDR operators]
\label{l:Fix}
Suppose that $A\cap B\neq \varnothing$ and let $L$ be an affine subspace of $X$ containing $A\cup B$. Set $\eta :=1 -\alpha +\alpha(1 -\lambda)(1 -\mu)$. Then the following hold:
\begin{enumerate}
\item 
\label{l:Fix_translation}
$\forall x\in X$, $\forall u\in (L-L)^\perp$, $P_A^\lambda(x +u) =P_A^\lambda x +(1 -\lambda)u$, $P_B^\mu(x +u) =P_B^\mu x +(1 -\mu)u$, and $T_{\lambda, \mu}^\alpha(x +u) =T_{\lambda, \mu}^\alpha x +\eta u$. 
\item
\label{l:Fix_projection}
$\forall c\in A\cap B$, $\forall u\in (L-L)^\perp$, $P_A^\lambda c =P_B^\mu c =c$, $P_A(c +u) =P_B(c +u)=c\in A\cap B$, and $T_{\lambda, \mu}^\alpha(c +u) =c +\eta u$.
\item  	
\label{l:Fix_inclusion}
 $A\cap B\subseteq \Fix T_{\lambda, \mu}^\alpha$ and, moreover, if $(1 -\lambda)(1 -\mu) =1$, then $(A\cap B)+ (L-L)^\perp \subseteq \Fix T_{\lambda, \mu}^\alpha$. In particular, $(A\cap B)+ (L-L)^\perp \subseteq \Fix T_{2, 2}^\alpha$.
\end{enumerate}
\end{lemma}	
\begin{proof}
\ref{l:Fix_translation}: Let $x\in X$, $u\in (L-L)^\perp$, and $w\in A\cap B$. Then $(L-L)^\perp= (L-w)^\perp$, hence $u\in (\aff A -w)^\perp$ and also $u\in (\aff B -w)^\perp$. Applying \cite[Lemma~3.2]{BLPW13a} implies that
\begin{equation}
P_A^\lambda(x +u) =(1 -\lambda)(x +u) +\lambda P_A(x +u) =(1 -\lambda)(x +u) +\lambda P_Ax =P_A^\lambda x +(1 -\lambda)u.
\end{equation}
Similarly, $P_B^\mu(x +u) =P_B^\mu x +(1 -\mu)u$ and the rest follows.

\ref{l:Fix_projection}: Let $c\in A\cap B$ and $u\in (L-L)^\perp$. Since $P_Ac =P_Bc =c$, we derive that $P_A^\lambda c =P_B^\mu c =c$, which yields $T_{\lambda, \mu}^\alpha c =c$. Combining with \ref{l:Fix_translation}, $P_A(c +u) =P_B(c +u) =c\in A\cap B$ and $T_{\lambda, \mu}^\alpha(c +u) =T_{\lambda, \mu}^\alpha c +\eta u =c +\eta u$.

\ref{l:Fix_inclusion}: It follows from \ref{l:Fix_projection} that $T_{\lambda, \mu}^\alpha c =c$ for all $c\in A\cap B$, which gives $A\cap B\subseteq \Fix T_{\lambda, \mu}^\alpha$. Now if $(1 -\lambda)(1 -\mu) =1$, then $\eta=1$, which leads to $(A\cap B)+ (L-L)^\perp \subseteq \Fix T_{\lambda, \mu}^\alpha$ due to \ref{l:Fix_translation}. 
\end{proof}

In the rest of this section, we assume that $X$ is a real Hilbert space and that $A$ and $B$ are convex but need not intersect. Then $B- A$ is convex, hence we can take $g:= P_{\overline{B-A}}0$ and set
\begin{equation}
E:= A\cap (B- g) \quad\text{and}\quad F:= (A +g)\cap B.
\end{equation}	
It is clear that if $A\cap B\neq \varnothing$, then $g= 0$ and $E= F= A\cap B$.
We also note that 
\begin{equation}
g\in B-A \iff E\neq \varnothing \iff F\neq \varnothing
\end{equation}
and from \cite[Lemma~2.2(i)\&(iv)]{BB94} that 
\begin{equation}
\label{e:gap}
a= P_Ab \text{~and~} b= P_Ba \quad\implies\quad g= b- a.
\end{equation}

Recall from \cite[Definition~4.1 and~4.33]{BC17} that a single-valued operator $T\colon X\to X$ is \emph{nonexpansive} if it is Lipschitz continuous with constant $1$, i.e.,
\begin{equation}
\forall x, y\in X,\quad \|Tx-Ty\|\leq \|x-y\|,
\end{equation}
and is \emph{$\alpha$-averaged} if $\alpha\in \left]0, 1\right[$ and $T= (1-\alpha)\Id+ \alpha R$ for some nonexpansive operator $R\colon X\to X$.
\begin{fact}
\label{f:cvxproj}
Let $C$ be a nonempty closed convex subset of $X$, let $x\in X$ and $p=P_Cx$. Then the following hold:
\begin{enumerate}
\item\label{f:cvxproj_ext} 
For every $\lambda\in\RP$, $P_C(p+\lambda(x-p))=P_C((1-\lambda)p+\lambda x)=p$.
\item\label{f:cvxproj_avg} 
For every $\lambda\in \left]0, 2\right[$, $P_C^\lambda$ is $\lambda/2$-averaged. Consequently, for every $\lambda\in \left]0, 2\right]$, $P_C^\lambda$ is nonexpansive. 
\end{enumerate}
\end{fact}
\begin{proof}
\ref{f:cvxproj_ext}: See \cite[Proposition~3.21]{BC17}. \ref{f:cvxproj_avg}: Combine Proposition~4.16, Corollary~4.41, Remark~4.34(i), and Corollary~4.18 in \cite{BC17}.
\end{proof}

Hereafter, whenever dealing with the harmonic-like quantity $\beta:=\big(\frac{1}{\beta_1} +\frac{1}{\beta_2}+\cdots +\frac{1}{\beta_k}\big)^{-1}$ of nonnegative numbers $\beta_i\in\RP$, we make a convention that $\beta=0$ if at least one $\beta_i$ equals $0$.
\begin{lemma}
\label{l:cont-avg}
Let $\lambda,\mu\in \left]0, 2\right]$. Then the following hold:
\begin{enumerate}
\item\label{l:cont-avg_cont}
For every $\alpha\in \RP$, $T_{\lambda,\mu}^\alpha$ is continuous and single-valued.
\item\label{l:cont-avg_avg}
For every $\alpha \in \big]0, 1+\hbeta\big[$ where $\hbeta:= \big(\frac{\lambda}{2-\lambda}+ \frac{\mu}{2-\mu}\big)^{-1}$, $T_{\lambda,\mu}^\alpha$ is $\alpha/(1+\hbeta)$-averaged.
\end{enumerate}
\end{lemma}
\begin{proof} 
By Fact~\ref{f:cvxproj}\ref{f:cvxproj_avg}, $P_A^\lambda$ and $P_B^\mu$ is nonexpansive, hence, continuous and single-valued. Thus, \ref{l:cont-avg_cont} follows.
To prove \ref{l:cont-avg_avg}, we first have that $P_B^\mu P_A^\lambda$ is nonexpansive. If $\lambda= 2$ or $\mu= 2$, then $\hbeta= 0$ and $\alpha \in \big]0, 1\big[$, so $T_{\lambda,\mu}^\alpha= (1-\alpha)\Id+ \alpha P_B^\mu P_A^\lambda$ is $\alpha$-averaged. 
If $\lambda, \mu< 2$, then $P_A^\lambda$ and $P_B^\mu$ are respectively $\lambda/2$- and $\mu/2$-averaged due to Fact~\ref{f:cvxproj}\ref{f:cvxproj_avg}. We derive from \cite[Proposition~4.44]{BC17} that $P_B^\mu P_A^\lambda$ is $\xi$-averaged with $\xi:= 2(\lambda+\mu-\lambda\mu)/(4-\lambda\mu)= 1/(1+\hbeta)$, and then from \cite[Proposition~4.40]{BC17} that $T_{\lambda,\mu}^\alpha$ is $\alpha/(1+\hbeta)$-averaged.
\end{proof}

\begin{lemma}[fixed points of convex gDR operators]
\label{l:cvxFix}
Let $\lambda,\mu\in \left]0, 2\right]$ and $\alpha\in\RPP$.  Then the following hold:
\begin{enumerate}
\item\label{l:cvxFix_inconsistent} 
If $E\neq \varnothing$ when $\min\{\lambda,\mu\}< 2$, and $A\cap B\neq \varnothing$ when $\lambda= \mu= 2$, then $P_A\Fix T_{\lambda,\mu}^\alpha\subseteq E$ and
\begin{equation}
\Fix T_{\lambda,\mu}^\alpha= \begin{cases}
E+ \frac{\mu}{\lambda+\mu-\lambda\mu}g& \text{~if~} \min\{\lambda,\mu\}< 2, \\
A\cap B+ N_{A-B}(0)& \text{~if~} \lambda= \mu= 2.
\end{cases}
\end{equation}
\item\label{l:cvxFix_intersect} 
If $A\cap B\neq\varnothing$ when $\min\{\lambda,\mu\}<2$, and $0\in \inte(B-A)$ when $\lambda= \mu= 2$, then $\Fix T_{\lambda,\mu}^\alpha= A\cap B$.
\item\label{l:cvxFix_reli} 
If $\reli A\cap\reli B\neq \varnothing$, then $\Fix T_{\lambda,\mu}^\alpha\cap \aff(A\cup B)= A\cap B$.
\end{enumerate}
\end{lemma}
\begin{proof}
\ref{l:cvxFix_inconsistent}:
First, it is straightforward to see that $\Fix T_{\lambda,\mu}^\alpha= \Fix P_B^\mu P_A^\lambda =\Fix T_{\lambda,\mu}^{1/2}$. If $\lambda= \mu= 2$, then the conclusion follows from \cite[Corollary~3.9]{BCL04}. 

Now assume that $\min\{\lambda,\mu\}< 2$. Since $\lambda,\mu\in \left]0, 2\right]$, we have $\lambda+\mu-\lambda\mu= 1-(1-\lambda)(1-\mu)> 0$.	
Let $x\in \Fix P_B^\mu P_A^\lambda$, that is, $x= P_B^\mu P_A^\lambda x$. Set $a= P_Ax$, $r= (1-\lambda)x+ \lambda a$, and $b= P_B r$. It follows that $x= (1-\mu)r+ \mu b$ and hence
\begin{equation}
\label{e:bax}
b= \tfrac{1}{\mu}x-\tfrac{1-\mu}{\mu}r= \tfrac{1}{\mu}x-\tfrac{1-\mu}{\mu}\Big((1-\lambda)x+\lambda a\Big)
= a+ \tfrac{\lambda+\mu-\lambda\mu}{\mu}(x-a).
\end{equation}
Since $\frac{\lambda+\mu-\lambda\mu}{\mu}> 0$, Fact~\ref{f:cvxproj}\ref{f:cvxproj_ext} implies that $P_A b = a$. Similarly, $P_B a=b$. We then use \eqref{e:gap} to deduce $g= b- a$. Combining with \eqref{e:bax} yields
\begin{equation}
x=a+\tfrac{\mu}{\lambda+\mu-\lambda\mu}(b-a)\in E+\tfrac{\mu}{\lambda+\mu-\lambda\mu}g.
\end{equation}
Conversely, take $x\in E+\frac{\mu}{\lambda+\mu-\lambda\mu}g$. It suffices to show that $x\in \Fix P_B^\mu P_A^\lambda$ and $P_Ax\in E$. By assumption, there exist $a\in A$ and $b\in B$ such that $a=b-g\in E$, $a=P_Ab$, $b=P_Ba$, and that $x=a+\frac{\mu}{\lambda+\mu-\lambda\mu}g=P_Bb+\frac{\mu}{\lambda+\mu-\lambda\mu}(b-P_Ab)$. Again, Fact~\ref{f:cvxproj}\ref{f:cvxproj_ext} yields $P_A x =P_A b=a\in E$. In turn, 
\begin{equation}
r:=P_A^\lambda x
=(1-\lambda)\Big(a+\tfrac{\mu}{\lambda+\mu-\lambda\mu}g\Big)+ \lambda a
=b-\tfrac{\lambda}{\lambda+\mu-\lambda\mu}g
=b+\tfrac{\lambda}{\lambda+\mu-\lambda\mu}(a-b)
\end{equation}
and, by Fact~\ref{f:cvxproj}\ref{f:cvxproj_ext}, $P_Br=b$. It follows that
\begin{equation}
P_B^\mu P_A^\lambda x= P_B^\mu r
=(1-\mu)\Big(b-\tfrac{\lambda}{\lambda+ \mu-\lambda\mu}g\Big)+\mu b
= a+g-\tfrac{(1-\mu)\lambda}{\lambda+\mu-\lambda\mu}g= x,
\end{equation}
which completes the proof.

\ref{l:cvxFix_intersect}: This follows from \ref{l:cvxFix_inconsistent} by noting that if $A\cap B\neq \varnothing$, then $g= 0$ and $E =A\cap B$, and from, e.g., \cite[Proposition~2.10(ii)]{BDNP16b} that if $0\in \inte(B-A)$, then $N_{A-B}(0)= 0$. 

\ref{l:cvxFix_reli}: If $\min\{\lambda,\mu\}<2$, then the conclusion follows from \ref{l:cvxFix_intersect}. If $\lambda=\mu=2$, then the conclusion follows from \cite[Proposition~4.1(iii)]{BNP15}.
\end{proof}

\begin{theorem}[convex possibly inconsistent case]
\label{t:cvg}
Let $\lambda,\mu\in \left]0, 2\right]$ and $\alpha \in \big]0, 1+\hbeta\big[$ where $\hbeta:= \big(\frac{\lambda}{2-\lambda}+ \frac{\mu}{2-\mu}\big)^{-1}$. Suppose that $E\neq \varnothing$ when $\min\{\lambda,\mu\}< 2$, and $A\cap B\neq \varnothing$ when $\lambda= \mu= 2$. 
Then the gDR sequence $(x_n)_\nnn$ generated by $T_{\lambda,\mu}^\alpha$ weakly converges to a point $\ox\in \Fix T_{\lambda,\mu}^\alpha$ with $P_A\ox\in E$.
\end{theorem}
\begin{proof}
According to Lemma~\ref{l:cont-avg}\ref{l:cont-avg_avg}, $T_{\lambda,\mu}^\alpha =(1 -\alpha)\Id +\alpha P_B^\mu P_A^\lambda$ is $\alpha/(1+\hbeta)$-averaged. Now use Lemma~\ref{l:cvxFix}\ref{l:cvxFix_inconsistent} and apply \cite[Proposition~5.16]{BC17} with all $\lambda_n= 1$.  
\end{proof}

In the light of Theorem~\ref{t:cvg}, if the gDR algorithm involves at most one reflection, then it is weakly convergent even in the inconsistent case.

\section{Quasi firm Fej\'er monotonicity}
\label{s:qfFm}
In this section, we further refine some results on quasi firm Fej\'er monotonicity, which were partly developed in \cite{DP16}.
Let $C$ and $U$ be nonempty subsets of $X$, let $\gamma \in \left[1, +\infty\right[$, and let $\beta \in \RP$. Recall from \cite[Definition~3.1]{DP16} that a set-valued operator $T\colon X\To X$ is \emph{$(C, \gamma, \beta)$-quasi firmly Fej\'er monotone} on $U$ if
\begin{equation}
\label{e:qfF}
\forall x \in U,\ \forall x_+ \in Tx,\ \forall \ox \in C,\quad
\|x_+ -\ox\|^2 +\beta\|x -x_+\|^2 \leq \gamma \|x -\ox\|^2.
\end{equation}
Here, when $\beta= 0$, we simply say that $T$ is \emph{$(C, \gamma)$-quasi Fej\'er monotone} on $U$.

\begin{lemma}[averaged quasi firmly Fej\'er monotone operators]
\label{l:averaged}
Let $C$ and $U$ be nonempty subsets of $X$, $\gamma\in \left[1, +\infty\right[$, $\beta\in \RP$, $\lambda\in \left]0, 1+\beta\right]$, and let $T\colon X\To X$ be a set-valued operator. Then the following are equivalent:
\begin{enumerate}
\item 
\label{l:averaged_ori}
$T$ is $(C, \gamma, \beta)$-quasi firmly Fej\'er monotone on $U$.
\item 
\label{l:averaged_avg}
$T^\lambda:= (1- \lambda)\Id+ \lambda T$ is $(C, 1- \lambda+ \lambda\gamma, \frac{1-\lambda+\beta}{\lambda})$-quasi firmly Fej\'er monotone on $U$.
\item 
\label{l:averaged_inv}
$T =\frac{\beta}{1+\beta}\Id+ \frac{1}{1+\beta}T^{1+\beta}$ with $T^{1+\beta}:= (1+ \beta)T -\beta\Id$ being $(C,\gamma+ \beta(\gamma- 1))$-quasi Fej\'er monotone on $U$.
\end{enumerate}
\end{lemma}
\begin{proof}
If \ref{l:averaged_ori} holds, then so does \ref{l:averaged_avg} due to \cite[Lemma~3.2]{DP16}. Conversely, if \ref{l:averaged_avg} holds, applying \cite[Lemma~3.2]{DP16} to $T= (1- \frac{1}{\lambda})\Id -\frac{1}{\lambda}T^\lambda$ and noting that $0< \frac{1}{\lambda}\leq 1+\frac{1-\lambda+\beta}{\lambda}$, we get \ref{l:averaged_ori}. So \ref{l:averaged_ori} and \ref{l:averaged_avg} are equivalent, which implies the equivalence of \ref{l:averaged_ori} and \ref{l:averaged_inv} by taking $\lambda =1+ \beta$.
\end{proof}

\begin{lemma}[composition of quasi firmly Fej\'er monotone operators]
\label{l:composition}	
Let $m$ be a positive integer, set $I:= \{1, \dots, m\}$, and for every $i\in I$, let $C_i$ and $U_i$ be a nonempty subset of $X$, $\gamma_i\in \left[1, +\infty\right[$, $\beta_i\in \RP$, and $T_i$ a $(C_i,\gamma_i,\beta_i)$-quasi firmly Fej\'er monotone operator on $U_i$. Set $C:= \bigcap_{i\in I} C_i$, $\gamma :=\gamma_1\cdots \gamma_m$,
\begin{equation}
\beta' :=\Big(\frac{1}{\beta_1\gamma_2\cdots \gamma_m} +\frac{1}{\beta_2\gamma_3\cdots \gamma_m} +\cdots +\frac{1}{\beta_{m-1}\gamma_m} +\frac{1}{\gamma_m}\Big)^{-1},
\quad\text{and~} \beta :=\Big(\sum_{i\in I} \frac{1}{\beta_i}\Big)^{-1}.
\end{equation}
Let $x_0, x_1,\dots, x_m$ be such that for every $i\in I$, $x_i\in T_ix_{i-1}$.
Then $\beta'\geq \beta\geq 0$ and, if $x_{i-1}\in U_i$ for every $i\in I$, it holds that
\begin{subequations}
\label{e:composition}	
\begin{align}
\forall \ox\in C,\quad \gamma\|x_0 -\ox\|^2 &\geq \|x_m -\ox\|^2 +\beta'\Big(\sum_{i\in I}\|x_{i-1} -x_i\|\Big)^2 \\
&\geq \|x_m -\ox\|^2 +\beta'\|x_0 -x_m\|^2.
\end{align}	
\end{subequations}
Consequently, if $T_iU_i\subseteq U_{i+1}$ for every $i\in I\smallsetminus \{m\}$, then $T_m\cdots T_1$ is $(C,\gamma,\beta')$- and also $(C,\gamma,\beta)$-quasi firmly Fej\'er monotone on $U_1$.	
\end{lemma}
\begin{proof}
Because $\gamma_i \geq 1$ and $\beta_i\geq 0$ for every $i\in I$, we have $\beta'\geq \beta\geq 0$. 
Next, let $\ox\in C$. For every $i\in I$, since $x_{i-1}\in U_i$ and $T_i$ is $(C,\gamma_i,\beta_i)$-quasi firmly Fej\'er monotone on $U_i$, we derive that 
\begin{subequations}
\begin{align}
\|x_1 -\ox\|^2 +\beta_1\|x_0-x_1\|^2 &\leq \gamma_1\|x_0 -\ox\|^2,\\
\|x_2 -\ox\|^2 +\beta_2\|x_1-x_2\|^2 &\leq \gamma_2\|x_1 -\ox\|^2,\\
&\;\;\vdots\\
\|x_m -\ox\|^2 +\beta_m\|x_{m-1}-x_m\|^2 &\leq \gamma_m\|x_{m-1} -\ox\|^2.
\end{align}
\end{subequations}
Using telescoping techniques yields
\begin{equation}
\label{e:telescope}
\begin{aligned}
(\gamma_1 \cdots \gamma_m)\|x_0 -\ox\|^2
\geq \|x_m -\ox\|^2 +\big(\beta_1\gamma_2\cdots \gamma_m\|x_0 -x_1\|^2 +\beta_2\gamma_3\cdots \gamma_m\|x_1 -x_2\|^2 \\
+\cdots +\beta_{m-1}\gamma_m\|x_{m-2} -x_{m-1}\|^2 +\beta_m\|x_{m-1}-x_m\|^2\big).
\end{aligned}
\end{equation}
By the coordinate version of Cauchy--Schwarz inequality,
\begin{equation}
\begin{aligned}
\Big(\frac{1}{\beta_1\gamma_2\cdots \gamma_m} +\cdots +\frac{1}{\beta_m}\Big)\big(\beta_1\gamma_2\cdots \gamma_m\|x_0 -x_1\|^2 +\cdots +\beta_m\|x_{m-1}-x_m\|^2\big) \\
\geq \Big(\sum_{i\in I}\|x_{i-1} -x_i\|\Big)^2.
\end{aligned}
\end{equation}
Combining with \eqref{e:telescope}, we obtain \eqref{e:composition}.

Now assume that $T_iU_i\subseteq U_{i+1}$ for every $i\in I\smallsetminus \{m\}$. Let $x\in U_1$ and $x_+\in (T_m\cdots T_1)x$. Then there exist $x_0, x_1,\dots, x_m$ such that $x_0 =x$, $x_m =x_+$, and $x_i\in Tx_{i-1}$ for every $i\in I$. We derive that $x_{i-1}\in U_i$ for every $i\in I$ and, by \eqref{e:composition},
\begin{equation}
\forall \ox\in C,\quad \gamma\|x -\ox\|^2 \geq \|x_+ -\ox\|^2 +\beta'\|x -x_+\|^2
\geq \|x_+ -\ox\|^2 +\beta\|x -x_+\|^2.
\end{equation}
The proof is complete.
\end{proof}

\begin{lemma}[quasi firm Fej\'er monotonicity of relaxed projectors]
\label{l:rproj}
Let $C$ be a nonempty subset of $X$ and $L$ be an affine subspace of $X$ containing $C$. Let also $w \in C$, $\varepsilon \in \left[0, 1\right[$, $\delta \in \RPP$, and $\lambda\in \left]0, 2\right]$. Set 
\begin{equation}
\label{e:setting}
\Omega :=C\cap \ball{w}{\delta},\quad
\gamma :=1 +\tfrac{\lambda\varepsilon}{1 -\varepsilon},
\quad\text{and}\quad \beta :=\tfrac{2 -\lambda}{\lambda}.
\end{equation}
Suppose that $C$ is $(\varepsilon, \delta)$-regular at $w$. Then $P_C^\lambda$ is $(\Omega+ (L-L)^\perp, \gamma, \beta)$-quasi firmly Fej\'er monotone and, in particular, $R_C$ is $(\Omega+ (L-L)^\perp, \frac{1+\varepsilon}{1-\varepsilon})$-quasi Fej\'er monotone on $\ball{w}{\delta/2}\cap L$.
\end{lemma}
\begin{proof} 
Let $x \in \ball{w}{\delta/2}\cap L\subseteq \ball{w}{\delta/2}$, let $p\in P_C x$, and let $\ox\in \Omega+ (L-L)^\perp$. By Fact~\ref{f:ImBall}\ref{f:ImBall_P}, $p\in \Omega$ . Writing $\ox =u +v$ with $u\in \Omega$ and $v\in (L-L)^\perp$, we note that $x, p, u\in L$, so $\scal{x -p}{v} =0$ and $\|p -\ox\|^2 =\|p -u\|^2 +\|v\|^2 \geq \|p -u\|^2$. Since $C$ is $(\varepsilon,\delta)$-regular at $w$ and $x -p \in \pnX{C}(p)$, we have
\begin{subequations}
\begin{align}
\scal{x -p}{p -\ox} &=\scal{x -p}{p -u}\geq -\varepsilon \|x -p\| \|p -u\| \\
&\geq -\tfrac{\varepsilon}{2}\big(\|x -p\|^2 +\|p -u\|^2\big)
\geq -\tfrac{\varepsilon}{2}\big(\|x -p\|^2 +\|p -\ox\|^2\big).
\end{align}
\end{subequations}
Therefore,
\begin{subequations}
\begin{align}
\|x-\ox\|^2&=\|x-p\|^2+\|p-\ox\|^2+2\scal{x-p}{p-\ox}\\
&\geq \|x-p\|^2+\|p-\ox\|^2-\varepsilon\big(\|x -p\|^2 +\|p -\ox\|^2\big)\\
&=(1-\varepsilon)\big(\|x-p\|^2+\|p-\ox\|^2\big),
\end{align}
\end{subequations}
which yields
\begin{equation}
\tfrac{1}{1-\varepsilon}\|x-\ox\|^2\geq \|x -p\|^2 +\|p -\ox\|^2,
\end{equation}
i.e., $P_C$ is $\big(\Omega+ (L-L)^\perp,\frac{1}{1-\varepsilon},1\big)$-quasi firmly Fej\'er monotone on $\ball{w}{\delta/2}\cap L$. In turn, Lemma~\ref{l:averaged} implies that $P^\lambda_C$ is $\big(\Omega+ (L-L)^\perp,\gamma,\beta\big)$-quasi firmly Fej\'er monotone on $\ball{w}{\delta/2}\cap L$ with $\gamma$ and $\beta$ as in \eqref{e:setting}.
In the case when $\lambda =2$, since $\gamma =\frac{1+\varepsilon}{1-\varepsilon}$ and $\beta =0$, the conclusion follows.
\end{proof}

\begin{proposition}[quasi firm Fej\'er monotonicity of gDR operators]
\label{p:qfF-DR}
Let $A$ and $B$ be closed subsets of $X$ such that $A\cap B\neq \varnothing$ and $L$ be an affine subspace of $X$ containing $A\cup B$. Let also $w \in A\cap B$, $\varepsilon_1 \in \left[0, 1/3\right]$, $\varepsilon_2 \in \left[0, 1\right[$, $\delta \in \RPP$, $\lambda, \mu \in \left]0, 2\right]$, and $\alpha \in \big]0, 1+\hbeta\big]$ where $\hbeta :=\big(\frac{\lambda}{2-\lambda} +\frac{\mu}{2-\mu}\big)^{-1}$.   
Set $\Omega :=A\cap B\cap\ball{w}{\delta}$,
\begin{equation}
\gamma :=1 -\alpha +\alpha\left(1 +\tfrac{\lambda\varepsilon_1}{1 -\varepsilon_1}\right)\left(1 +\tfrac{\mu\varepsilon_2}{1 -\varepsilon_2}\right),
\quad\text{and}\quad \beta :=\tfrac{1-\alpha+\hbeta}{\alpha}.
\end{equation}
Suppose that $A$ and $B$ are $(\varepsilon_1, \delta)$- and $(\varepsilon_2, \sqrt{2}\delta)$-regular at $w$, respectively.
Then the following hold:
\begin{enumerate}
\item 
\label{p:qfF-DR_L}
$T_{\lambda, \mu}^\alpha$ is $\big(\Omega+ (L-L)^\perp, \gamma, \beta\big)$-quasi firmly Fej\'er monotone on $\ball{w}{\delta/2}\cap L$.
\item 
\label{p:qfF-DR_X}
$T_{\lambda, \mu}^\alpha$ is $\big(\Omega, \gamma, \beta\big)$-quasi firmly Fej\'er monotone on $\ball{w}{\delta/2}$.
\item 
\label{p:qfF-DR_22}
$T_{2, 2}^\alpha$ is $\big(\Omega+ (L-L)^\perp, \gamma, \beta\big)$-quasi firmly Fej\'er monotone on $\ball{w}{\delta/2}$.
\end{enumerate}	
\end{proposition}
\begin{proof} 
\ref{p:qfF-DR_L}: We first derive from Lemma~\ref{l:rproj} that $P_A^\lambda$ is $\big(\Omega+ (L-L)^\perp, \gamma_1, \beta_1\big)$-quasi firmly Fej\'er monotone on $\ball{w}{\delta/2}\cap L$ and that $P_B^\mu$ is $\big(\Omega+ (L-L)^\perp, \gamma_2, \beta_2\big)$-quasi firmly Fej\'er monotone on $\ball{w}{\delta/\sqrt{2}}\cap L$, where
\begin{equation}
\gamma_1 :=1 +\tfrac{\lambda\varepsilon_1}{1 -\varepsilon_1},\
\gamma_2 :=1 +\tfrac{\mu\varepsilon_2}{1 -\varepsilon_2},\
\beta_1 :=\tfrac{2 -\lambda}{\lambda},\ 
\text{~and~} \beta_2 :=\tfrac{2 -\mu}{\mu}.
\end{equation} 
Now let $x\in \ball{w}{\delta/2}\cap L$. On the one hand, Fact~\ref{f:ImBall}\ref{f:ImBall_reg} yields $P_A^\lambda x \subseteq \ball{w}{\delta/\sqrt{2}}$. On the other hand, $P_A^\lambda x =(1 -\lambda)x +\lambda P_Ax \subseteq L$ since $x\in L$, $P_Ax\subseteq A\subseteq L$, and $L$ is an affine subspace. We deduce that $P_A^\lambda(\ball{w}{\delta/2}\cap L)\subseteq \ball{w}{\delta/\sqrt{2}}\cap L$. Noting that $\hbeta =(\frac{1}{\beta_1} +\frac{1}{\beta_1})^{-1}$, we apply Lemma~\ref{l:composition} to $(P_A^\lambda, P_B^\mu)$ to obtain the $(\Omega+ (L-L)^\perp, \gamma_1\gamma_2, \hbeta)$-quasi firm Fej\'er monotonicity of $P_B^\mu P_A^\lambda$ on $\ball{w}{\delta/2}\cap L$. The conclusion follows from Lemma~\ref{l:averaged}\ref{l:averaged_ori}--\ref{l:averaged_avg} applied to the operators $P_B^\mu P_A^\lambda$ and $T_{\lambda,\mu}^\alpha = (1-\alpha)\Id+\alpha P_B^\mu P_A^\lambda$.

\ref{p:qfF-DR_X}: Apply \ref{p:qfF-DR_L} with $L =X$.

\ref{p:qfF-DR_22}: Let $x\in \ball{w}{\delta/2}$ and $x_+\in T_{2,2}^\alpha x$. Define $y =P_Lx$ and $y_+ =P_Lx_+$. Then by Lemma~\ref{l:projL}\ref{l:projL_ball}, $y\in \ball{w}{\delta/2}\cap L$ and by Lemma~\ref{l:shadow}\ref{l:shadow_step}--\ref{l:shadow_diff}, $y_+\in  T_{2,2}^\alpha y$ and $v :=x_+ -y_+ =x -y$. Applying \ref{p:qfF-DR_L} to $T_{2,2}^\alpha$ yields
\begin{equation}
\label{e:L-qfF}
\forall \overline{y}\in \Omega+ (L-L)^\perp,\quad 
\|y_+ -\overline{y}\|^2 +\beta\|y -y_+\|^2\leq \gamma\|y -\overline{y}\|^2.
\end{equation}
Now let $\ox\in \Omega+ (L-L)^\perp$. From Lemma~\ref{l:projL}\ref{l:projL_perp}, we observe that $v\in (L -L)^\perp$, hence $\ox -v\in \Omega+ (L-L)^\perp$. Substituting $\overline{y} =\ox -v =\ox -(x_+ -y_+) =\ox -(x -y)$ into \eqref{e:L-qfF} completes the proof.
\end{proof}

\section{Quasi coercivity}
\label{s:qcoer}

Let $C$ and $U$ be nonempty subsets of $X$ and let $\nu\in \RPP$. Recall from \cite[Definition~3.3]{DP16} that an operator $T\colon X\To X$ is \emph{$(C, \nu)$-quasi coercive} on $U$ if
\begin{equation}
\forall x\in U,\ \forall x_+\in Tx,\quad \|x -x_+\|\geq \nu d_C(x),
\end{equation}
and \emph{$C$-quasi coercive} around $w\in X$ if it is $(C, \nu)$-quasi coercive on $\ball{w}{\delta}$ for some $\nu\in \RPP$ and $\delta\in \RPP$.

In this section, we show quasi coercivity of gDR operators under different assumptions on the system of sets. In particular, under \emph{affine-hull regularity} assumption, Proposition~\ref{p:coer-DR} improves some existing results on quasi coercivity (see Remark~\ref{r:imprv_qcoer}), while under \emph{linear regularity} assumption and some parameter restriction, Proposition~\ref{p:coer-DR-lin} proves the quasi coercivity of gDR operators.

\begin{lemma}[averaged quasi coercive operators]
\label{l:averaged-coer}
Let $C$ and $U$ be nonempty subsets of $X$, $\nu\in \RPP$, $\lambda \in \RPP$, and let $T\colon X\To X$ be a set-valued operator. Then $T$ is $(C, \nu)$-quasi coercive on $U$ if and only if $T^\lambda :=(1 -\lambda)\Id +\lambda T$ is $(C, \lambda\nu)$-quasi coercive on $U$.
\end{lemma}
\begin{proof}
Assume that $T$ is $(C, \nu)$-quasi coercive on $U$. Let $x\in U$ and let $x_+\in T^\lambda x$. 
Then there exists $s\in Tx$ such that $x_+ =(1 -\lambda)x +\lambda s$. We obtain that 
$\|x -x_+\| =\lambda\|x -s\| \geq\lambda\nu d_C(x)$, and $T^\lambda$ is thus $(C, \lambda\nu)$-quasi coercive on $U$. Conversely, note that $T =(1 -\frac{1}{\lambda})\Id +\frac{1}{\lambda}T^\lambda$.
\end{proof}

\begin{lemma}[global quasi coercivity]
\label{l:global-coer}
Let $C$ be a nonempty subset of $X$, $L$ a closed subset of $X$, and $T\colon X\to X$ a continuous single-valued operator. Suppose that $\Fix T\cap L\subseteq C$ and that, for every $w\in C$, $T$ is $C$-quasi coercive on $\ball{w}{\delta}\cap L$ for some $\delta\in \RPP$. Then $T$ is $C$-quasi coercive on $S\cap L$ for every bounded set $S$ of $X$.
\end{lemma}
\begin{proof}
Let $S$ be a bounded set of $X$ and suppose on the contrary that $T$ is not $C$-quasi coercive on $S\cap L$. Then there exist sequences $\varepsilon_n\downarrow 0$ and $x_n\in S\cap L$ such that
\begin{equation}
\forall\nnn,\quad 0\leq \|x_n- T x_n\|< \varepsilon_n d_{C}(x_n).
\end{equation}
Since $(x_n)_\nnn$ is bounded, so is $(d_{C}(x_n))_\nnn$, and hence $x_n- T x_n\to 0$. 
Extracting a convergent subsequence without relabeling, we can assume $x_n\to \ox$. It follows from the continuity of $T$ and the closedness of $L$ that $\ox\in \Fix T\cap L\subseteq C$. In turn, there exist $\nu\in \RPP$ and $\delta\in \RPP$ such that $T$ is $(C, \nu)$-quasi coercive on $\ball{\ox}{\delta}\cap L$. Thus, for all $n$ sufficiently large,
\begin{equation}
\varepsilon_n d_{C}(x_n)> \|x_n- T x_n\|\geq \nu d_{C}(x_n),
\end{equation}
which is a contradiction since $\varepsilon_n\downarrow 0$ and $d_{C}(x_n)> 0$.
\end{proof}

\subsection{In the presence of affine-hull regularity}

In this section, we aim to improve the estimate for quasi coercivity constant previously obtained in \cite{DP16}.
To proceed, we need the following technical lemma. 

\begin{lemma}
\label{l:const}	
Let $\theta\in \left[0, 1\right[$, $\mu\in \left]0, 2\right]$, and let $u,v\in X$ be such that $\scal{u}{v}\geq -\theta\|u\|\|v\|$. Then
\begin{equation}
\|u +\mu v\|^2\geq \max\Big\{\mu^2(1-\theta^2)\|v\|^2, \frac{4\mu^2(1-\theta^2)}{\big(|1-\mu| +\sqrt{(1-\mu)^2 +4\mu(1-\theta^2)}\big)^2}\|u+v\|^2\Big\}.
\end{equation}
\end{lemma}
\begin{proof}
First, it follows from $\scal{u}{v}\geq -\theta\|u\|\|v\|$ that
\begin{subequations}
\begin{align}
\|u +\mu v\|^2
&\geq \|u\|^2+ \mu^2\|v\|^2 -2\mu\theta\|u\| \|v\| \\
&=(\|u\|- \mu\theta\|v\|)^2+ \mu^2(1-\theta^2)\|v\|^2 
\geq \mu^2(1-\theta^2)\|v\|^2.
\end{align}
\end{subequations} 
Next, we show that 
\begin{equation}\label{e:170813a}
\|u +\mu v\|^2\geq \xi\|u+v\|^2,
\quad\text{where~} \xi :=\frac{4\mu^2(1-\theta^2)}{\big(|1-\mu| +\sqrt{(1-\mu)^2 +4\mu(1-\theta^2)}\big)^2}.
\end{equation}
Observe that $|1-\mu|+\sqrt{(1-\mu)^2 +4\mu(1-\theta^2)}\geq \sqrt{4\mu(1-\theta^2)}> 0$, so
\begin{equation}
\xi= \frac{4\mu^2(1-\theta^2)}{\big(|1-\mu| +\sqrt{(1-\mu)^2 +4\mu(1-\theta^2)}\big)^2}\leq \frac{4\mu^2(1-\theta^2)}{4\mu(1-\theta^2)} =\mu.
\end{equation}
Now \eqref{e:170813a} is equivalent to
\begin{equation}
(1-\xi)\|u\|^2+(\mu^2-\xi)\|v\|^2+2(\mu-\xi)\scal{u}{v}\geq 0.
\end{equation}
Since $\mu-\xi\geq 0$ and $\scal{u}{v}\geq -\theta\|u\|\|v\|$, it suffices to prove that
\begin{equation}
\forall u\in X,\ \forall v\in X,\quad (1-\xi)\|u\|^2+(\mu^2-\xi)\|v\|^2-2\theta(\mu-\xi)\|u\|\|v\|\geq 0.
\end{equation}
The latter one can be written as
\begin{equation}
\forall u\in X,\ \forall v\in X,\quad 
\left[\begin{matrix}
\|u\| & \|v\|
\end{matrix}\right]
M
\left[\begin{matrix}
\|u\| \\ \|v\|
\end{matrix}\right]\geq 0 
\quad\text{with}\quad 
M :=\left[\begin{matrix}
1-\xi & \theta(\xi-\mu)\\
\theta(\xi-\mu) & \mu^2-\xi
\end{matrix}\right],
\end{equation}
which is equivalent to positive semidefiniteness of $M$.
Because $M$ is a symmetric $2\times 2$ matrix whose trace $(1-\xi) +(\mu^2-\xi) =(1-\mu)^2 +2(\mu-\xi)\geq 0$, 
it is positive semidefinite if and only if its determinant
\begin{equation}\label{e:170813b}
(1-\xi)(\mu^2-\xi) -\theta^2(\mu-\xi)^2
=(1-\theta^2)\xi^2 -(1+\mu^2-2\mu\theta^2)\xi +\mu^2(1-\theta^2)\geq 0.
\end{equation}
Finally, one can directly check that $\xi$ in \eqref{e:170813a} is a solution of \eqref{e:170813b}. The proof is complete.
\end{proof}

\begin{lemma}
\label{l:Lcoer}
Let $A$ and $B$ be closed subsets of $X$ with $A\cap B\neq \varnothing$, $U$ a nonempty subset of $X$, $\lambda,\mu\in \left]0, 2\right]$, $\alpha \in \RPP$, and $\nu\in \RPP$. Suppose that $T_{\lambda, \mu}^\alpha$ is $(A\cap B, \nu)$-quasi coercive on $U\cap L$ with $L:= \aff(A\cup B)$. 
Then $T_{\lambda, \mu}^\alpha$ is $\big((A\cap B)+ (L-L)^\perp, \nu\big)$-quasi coercive on $P_L^{-1}(U\cap L)$. Moreover, if $\min\{\lambda,\mu\}<2$, then $T_{\lambda, \mu}^\alpha$ is also $(A\cap B, \nu')$-quasi coercive on $P_L^{-1}(U\cap L)$ with $\nu':=\min\{\nu,\alpha(\lambda+\mu-\lambda\mu)\}$.
\end{lemma}
\begin{proof}
Let $x\in P_L^{-1}(U\cap L)$ and $x_+\in T_{\lambda,\mu}^{\alpha}x$. Define $y =P_Lx$ and $y_+ =P_Ly$, then $y\in U\cap L$ and, by Lemma~\ref{l:shadow}\ref{l:shadow_step}, $y_+\in T_{\lambda,\mu}^{\alpha}y$ . By assumption,
\begin{equation}\label{e:170716b}
\|y-y_+\|\geq \nu d_{A\cap B}(y).
\end{equation}
Setting $\eta :=(1-\alpha)+\alpha(1-\lambda)(1-\mu)$, we obtain from Lemma~\ref{l:shadow}\ref{l:shadow_dist} that
\begin{equation}\label{e:170716c}
\|x -x_+\|^2 =\|y -y_+\|^2 +(1-\eta)^2\|x -y\|^2\geq \|y -y_+\|^2.
\end{equation}
By Lemma~\ref{l:projL}\ref{l:projL_perp}, $x -y =x -P_Lx\in (L-L)^\perp$ and then, by Lemma~\ref{l:aff}\ref{l:aff_proj}, 
\begin{equation}
d_{A\cap B}(y)=d_{(A\cap B)+(L-L)^\perp}(y)
=d_{(A\cap B)+(L-L)^\perp}(y+(x-y))
=d_{(A\cap B)+(L-L)^\perp}(x).
\end{equation}
This together with \eqref{e:170716b} and \eqref{e:170716c} yields
\begin{equation}
\|x -x_+\|\geq \|y -y_+\|\geq \nu d_{A\cap B}(y)
=\nu d_{(A\cap B)+(L-L)^\perp}(x),
\end{equation}
i.e., $T_{\lambda, \mu}^\alpha$ is $\big((A\cap B)+ (L-L)^\perp,\nu\big)$-quasi coercive on $P_L^{-1}(U\cap L)$.

Now assume that $\min\{\lambda,\mu\}<2$. Then $\eta<1$. Combining \eqref{e:170716b}, \eqref{e:170716c}, and Lemma~\ref{l:aff}\ref{l:aff_dist}, we deduce that
\begin{subequations}
\begin{align}
\|x -x_+\|^2&\geq \nu^2 d_{A\cap B}^2(y) +(1-\eta)^2 d_L^2(x) \\
&\geq \min\{\nu^2,(1-\eta)^2\}\big(d_{A\cap B}^2(y) +d_L^2(x)\big)
=(\nu')^2 d_{A\cap B}^2(x),
\end{align}
\end{subequations}
which means that $T_{\lambda,\mu}^{\alpha}$ is $(A\cap B,\nu')$-quasi coercive on $P_L^{-1}(U\cap L)$.
\end{proof}

\begin{proposition}[quasi coercivity of gDR operators under affine-hull regularity]
\label{p:coer-DR}
Let $A$ and $B$ be closed subsets of $X$ such that $A\cap B\neq \varnothing$. 
Let also $w \in A\cap B$, $\varepsilon \in \left[0, 1/3\right]$,  $\delta \in \RPP$, $\kappa\in\RPP$, $\lambda,\mu\in \left]0, 2\right]$, and $\alpha \in \RPP$. 
Suppose that $A$ is $(\varepsilon, \delta)$-regular at $w$, that $\{A, B\}$ is $\kappa$-linearly regular on $\ball{w}{\delta/2}$, and that the CQ-number $\theta :=\theta_{A,B,L}(w,\sqrt{2}\delta)< 1$ with $L :=\aff(A\cup B)$. 
Define
\begin{equation}
\nu:= \frac{\alpha\sqrt{1 -\theta^2}}{\kappa}\min\Big\{\lambda, \frac{2\mu}{|1-\mu|+\sqrt{(1-\mu)^2+4\mu(1-\theta^2)}}\Big\}.
\end{equation}
Then the following hold:
\begin{enumerate}
\item\label{p:coer-DR_L} 
$T_{\lambda, \mu}^\alpha$ is $\big(A\cap B,\nu\big)$-quasi coercive on $\ball{w}{\delta/2}\cap L$.
\item\label{p:coer-DR_gen}
$T_{\lambda, \mu}^\alpha$ is $\big((A\cap B)+ (L-L)^\perp,\nu\big)$-quasi coercive on $\ball{w}{\delta/2}$.
\item\label{p:coer-DR_less}
If $\min\{\lambda,\mu\}<2$, then $T_{\lambda, \mu}^\alpha$ is $(A\cap B,\nu')$-quasi coercive on $\ball{w}{\delta/2}$ with
$\nu':=\min\{\nu,\alpha(\lambda+\mu-\lambda\mu)\}$.
\end{enumerate}
Additionally, if $\lambda\neq 1$, $A$ is an affine subspace of $X$, and $\{A, B\}$ is $\kappa$-linearly regular on $\ball{w}{\sqrt{2}\delta/2}$, then we can choose
\begin{equation}
\nu= \frac{\alpha\sqrt{1-\theta^2}}{\kappa}\min\{\lambda, \mu|1-\lambda|\}.
\end{equation}
\end{proposition}
\begin{proof} 
First, it follows from the definition of the CQ-number $\theta$ that 
\begin{equation}
\label{e:strreg}
\left.\begin{aligned}
&a \in A\cap \ball{w}{\sqrt{2}\delta},\ b \in B\cap \ball{w}{\sqrt{2}\delta}, \\
&u \in \pnX{A}(a)\cap (L-a),\ v \in \pnX{B}(b)\cap (L-b)
\end{aligned}\right\}
\implies
\scal{u}{v}\geq -\theta\|u\|\|v\|.
\end{equation}
To prove \ref{p:coer-DR_L}, in view of Lemma~\ref{l:averaged-coer}, it suffices to prove that $P_B^\mu P_A^\lambda$ is $(A\cap B, \nu/\alpha)$-quasi coercive on $\ball{w}{\delta/2}\cap L$.  
Let $x\in \ball{w}{\delta/2}\cap L$ and $s\in P_B^\mu P_A^\lambda x$. Then there exist $a \in P_Ax$, $r\in P_A^\lambda x$, and $b \in P_Br$ such that 
\begin{equation}
x -r =\lambda(x -a) \quad\text{and}\quad r -s =\mu(r -b).
\end{equation}
We note that $r\in P_A^\lambda(L)\subseteq L$ due to Lemma~\ref{l:aff}\ref{l:aff_imag}. 
By Fact~\ref{f:ImBall}\ref{f:ImBall_P}, $a \in P_Ax \subseteq A\cap \ball{w}{\delta}$. Since $\varepsilon \in \left[0, 1/3\right]$, Fact~\ref{f:ImBall}\ref{f:ImBall_reg} yields $r \in \ball{w}{\sqrt{2}\delta/2}$.
Again by Fact~\ref{f:ImBall}\ref{f:ImBall_P}, $b \in P_Br \subseteq B\cap \ball{w}{\sqrt{2}\delta}$.
Now as $x -r =\lambda(x -a) \in \pnX{A}(a)\cap(L-a)$ and $r -s =\mu(r -b) \in \pnX{B}(b)\cap(L-b)$, it follows from \eqref{e:strreg} that
\begin{equation}
\scal{x -a}{r -s} \geq -\theta\|x -a\|\|r -s\|
\quad\text{and}\quad
\scal{x -r}{r -b} \geq -\theta\|x -r\|\|r -b\|.
\end{equation}
Applying Lemma~\ref{l:const} with $(u, v)= (r-s, x-a)$, we obtain
\begin{subequations}\label{e:160802a}
\begin{align}
\|x-s\|^2&= \|(x-r)+ (r-s)\|^2= \|\lambda(x-a)+ (r-s)\|^2 \\
&\geq \lambda^2(1-\theta^2)\|x-a\|^2
=\lambda^2(1-\theta^2)d_A^2(x);
\end{align}
\end{subequations}
and with $(u, v)= (x-r, r-b)$, we obtain
\begin{subequations}\label{e:160802b}
\begin{align}
\|x -s\|^2= \|(x -r) +\mu(r -b)\|^2 
&\geq \max\{\mu^2(1-\theta^2)\|r-b\|^2, \xi\|x -b\|^2\}\\
&\geq \max\{\mu^2(1-\theta^2)d_B^2(r), \xi d_B^2(x)\},
\end{align}
\end{subequations}
where $\xi:=\tfrac{4\mu^2(1-\theta^2)}{\big(|1-\mu| +\sqrt{(1-\mu)^2 +4\mu(1-\theta^2)}\big)^2}$. 
Combining with the $\kappa$-linear regularity of $\{A, B\}$ yields
\begin{subequations}
\begin{align}
\|x -s\| &\geq 
\min\Big\{\lambda\sqrt{1-\theta^2}, \sqrt{\xi}\Big\}
\max\{d_A(x), d_B(x)\}\\
&\geq
\frac{\sqrt{1-\theta^2}}{\kappa}
\min\bigg\{\lambda,
\frac{2\mu}{|1-\mu| +\sqrt{(1-\mu)^2 +4\mu(1-\theta^2)}} \bigg\}
d_{A\cap B}(x) =\frac{\nu}{\alpha} d_{A\cap B}(x),
\end{align}
\end{subequations}
and \ref{p:coer-DR_L} is proved.
Now applying Lemma~\ref{l:Lcoer} and noting from Lemma~\ref{l:projL}\ref{l:projL_ball} that $\ball{w}{\delta/2}\subseteq P_L^{-1}(\ball{w}{\delta/2}\cap L)$, we get \ref{p:coer-DR_gen} and \ref{p:coer-DR_less}.

In addition, suppose that $\lambda\neq 1$, that $A$ is an affine subspace of $X$, and that $\{A,B\}$ is $\kappa$-linearly regular on $\ball{w}{\sqrt{2}\delta/2}$. Then $x-a= \frac{1}{1-\lambda}(r-a)$ and, by \eqref{e:160802a},
\begin{equation}
\|x-s\|^2\geq \lambda^2(1-\theta^2)\|x-a\|^2
= \frac{\lambda^2(1-\theta^2)}{(1-\lambda)^2}\|r-a\|^2
\geq \frac{\lambda^2(1-\theta^2)}{(1-\lambda)^2} d_A^2(r).
\end{equation}
Together with \eqref{e:160802b} and the $\kappa$-linear regularity of $\{A, B\}$ on $\ball{w}{\sqrt{2}\delta/2}$, we get
\begin{subequations}
\label{e:dArdBr}
\begin{align}
\|x -s\|&\geq \sqrt{1-\theta^2}\min\Big\{\frac{\lambda}{|1-\lambda|}, \mu\Big\} \max\{d_A(r), d_B(r)\} \\
&\geq \frac{\sqrt{1-\theta^2}}{\kappa}\min\Big\{\frac{\lambda}{|1-\lambda|}, \mu\Big\} d_{A\cap B}(r).
\end{align}
\end{subequations}
By applying Lemma~\ref{l:aff}\ref{l:aff_dist} (with $C=A\cap B$ and $L=A$), 
\begin{subequations}
\begin{align}
d_{A\cap B}^2(r)&=d_{A\cap B}^2(P_A x)+(1-\lambda)^2 d_A^2(x)\\
&\geq (1-\lambda)^2\big(d_{A\cap B}^2(P_Ax)+d_A^2(x)\big)
=(1-\lambda)^2 d_{A\cap B}^2(x),
\end{align}
\end{subequations}
and the conclusion follows.
\end{proof}

\begin{corollary}
\label{c:coer-DR}
Let $A$ and $B$ be closed subsets of $X$ such that $A\cap B\neq \varnothing$. Let also $w \in A\cap B$, $\lambda,\mu\in \left]0, 2\right]$, and $\alpha \in \RPP$. Suppose that $A$ is superregular at $w$ and $\{A, B\}$ is affine-hull regular at $w$. Define $L :=\aff(A\cup B)$. 
Then the following hold:
\begin{enumerate}
\item\label{c:coer-DR-L}
$T_{\lambda, \mu}^\alpha$ is $(A\cap B)$-quasi coercive on $\ball{w}{\delta}\cap L$ for some $\delta\in \RPP$.
\item\label{c:coer-DR_gen}
$T_{\lambda, \mu}^\alpha$ is $\big((A\cap B)+ (L-L)^\perp\big)$-quasi coercive around $w$.
\item\label{c:coer-DR_less}
If $\min\{\lambda,\mu\}<2$, then $T_{\lambda, \mu}^\alpha$ is $(A\cap B)$-quasi coercive around $w$.
\end{enumerate}
\end{corollary}
\begin{proof}
By superregularity, affine-hull regularity, Proposition~\ref{p:Lstr}\ref{p:Lstr_lin}, and Proposition~\ref{p:str&CQ}, we can find $\varepsilon\in \left[0, 1/3\right]$, $\delta\in \RPP$, and $\kappa\in \RPP$ such that $A$ is $(\varepsilon, \delta)$-regular at $w$, that $\{A, B\}$ is $\kappa$-linearly regular on $\ball{w}{\delta/2}$, and that the CQ-number $\theta_{A,B,L}(w, \sqrt{2}\delta)< 1$. The conclusion then follows from Proposition~\ref{p:coer-DR}.
\end{proof}

\begin{remark}
In Proposition~\ref{p:coer-DR} and Corollary~\ref{c:coer-DR}, the term $(L -L)^\perp$ is indeed indispensable when $\lambda =\mu =2$. For instance, suppose $A$ and $B$ are two arbitrary closed sets in $X =\RR^3$ with $w\in A\cap B$ and $L =\aff(A\cup B) =\RR^2\times \{0\}$. Then $(L-L)^\perp =\menge{(0,0,\xi)}{\xi\in \RR}$. For any $\nu,\delta\in \RPP$, taking $x =w +(0,0,\delta)$, we have $x_+ =x =T_{2,2}^\alpha x$ due to Lemma~\ref{l:Fix}\ref{l:Fix_projection}, and $d_{A\cap B}(x) =d_L(x) =\delta$. Hence, $0 =\|x_+-x\|< \nu d_{A\cap B}(x) =\nu\delta$, i.e., $T_{2,2}^\alpha$ fails to be $(A\cap B,\nu)$-quasi coercive on $\ball{w}{\delta}$.
\end{remark}

\begin{remark}[improved quasi coercivity constant for DR operator]\label{r:imprv_qcoer}
As we will see in the next part (see Remark~\ref{r:better_rate}), larger quasi coercivity constant will improve the linear convergence rate.
In this aspect, Proposition~\ref{p:coer-DR} indeed subsumes \cite[Lemma~4.2]{Pha14} and \cite[Lemma~3.14]{HL13}; and moreover, provides \emph{larger (local) quasi coercivity constant} for the classical DR operator ($\lambda=\mu=2$ and $\alpha=1/2$). To see this, we consider the DR operator $T:=T_{2,2}^{1/2}$ for the pair $(A,B)$ of closed subsets of $X$ and $w\in A\cap B$. Assume that $A$ is superregular at $w$ and that $\{A,B\}$ is strongly regular at $w$, i.e., the limiting CQ-number $\overline{\theta}_{A,B}(w)<1$ (due to Proposition~\ref{p:str&CQ}).
Then for any $\theta\in\big]\overline{\theta},1\big[$, there exist $\varepsilon\in\left[0,1/4\right[$, $\delta\in\RPP$, and $\kappa\in\RP$ such that
\begin{enumerate}
\item $A$ is $(\varepsilon,2\delta)$-regular at $w$;
\item $\theta$ is the CQ-number at $w$ associated with $(A,B)$ and $3\delta$, which implies that
\begin{equation}
\left\{\begin{aligned}
&a\in A\cap\ball{w}{3\delta},b\in B\cap\ball{w}{3\delta},\\
&u\in \pnX{A}(a),v\in\pnX{B}(b)
\end{aligned}\right.
\implies
\scal{u}{v}\geq -\theta\|u\|.\|v\|;
\end{equation}
\item $\{A,B\}$ is $\kappa$-linearly regular on $\ball{w}{2\delta}$.
\end{enumerate}
Now, on the one hand, \cite[Lemma~4.2]{Pha14} derives that $T$ is $(A\cap B,\hat{\nu})$-quasi coercive on $\ball{w}{\delta}$ with $\hat{\nu}=\frac{\sqrt{1-\theta}}{\kappa\sqrt{5}}$.
On the other hand, Proposition~\ref{p:coer-DR} derives that $T$ is $(A\cap B,\nu)$-quasi coercive on $\ball{w}{\delta}$ with $\nu=\frac{\sqrt{1-\theta^2}}{\kappa}\frac{2}{1+\sqrt{1+8(1-\theta^2)}}$. It is clear that $\nu\geq\frac{\sqrt{1-\theta^2}}{\kappa}\frac{2}{1+\sqrt{9}}>
\frac{\sqrt{1-\theta}}{\kappa}\frac{1}{\sqrt{5}}=\hat{\nu}$ regardless of $\theta\in\left]0,1\right[$.

In addition, if $A$ is an affine subspace, then Proposition~\ref{p:coer-DR} also improves the estimate in \cite[Lemma~3.14]{HL13}. Indeed, under the assumptions made, on the one hand, \cite[Lemma~3.14]{HL13} implies that $T$ is $(A\cap B,\hat{\nu})$-quasi coercive on $\ball{w}{\delta}$ with $\hat{\nu}=\frac{\sqrt{1-\theta}}{\kappa}$. On the other hand, Proposition~\ref{p:coer-DR} yields that $T$ is $(A\cap B,\nu)$-quasi coercive on $\ball{w}{\delta}$ with $\nu=\frac{\sqrt{1-\theta^2}}{\kappa}>\hat{\nu}$.

Similarly, we can show that Proposition~\ref{p:coer-DR} also improves quasi coercivity constant derived in \cite[Proposition~3.8]{DP16} for gDR operators.
\end{remark}

The following global version of Proposition~\ref{p:coer-DR} is an extension of \cite[Lemma~4.3]{BNP15}.
\begin{proposition}[global quasi coercivity of convex gDR operators]
\label{p:gcoer-DR}
Let $\lambda,\mu\in \left]0,2\right]$ and $\alpha\in \RPP$. 
Then the following hold:
\begin{enumerate}
\item\label{p:gcoer-DR_reli}
If $A$ and $B$ are closed convex subsets of $X$ such that $\reli A\cap \reli B\neq \varnothing$, then $T_{\lambda, \mu}^\alpha$ is $\big((A\cap B)+ (L-L)^\perp\big)$-quasi coercive on every bounded set $S$ of $X$ with $L :=\aff(A\cup B)$, and if additionally $\min\{\lambda,\mu\}<2$, then $T_{\lambda, \mu}^\alpha$ is $(A\cap B)$-quasi coercive on every bounded set $S$ of $X$.
\item\label{p:gcoer-DR_poly} 
If $A$ and $B$ are polyhedral subsets of $X$ such that $A\cap B\neq \varnothing$, then $T_{\lambda, \mu}^\alpha$ is $\Fix T_{\lambda, \mu}^\alpha$-quasi coercive on every bounded set $S$ of $X$. 
\end{enumerate}
\end{proposition}
\begin{proof}
First, by Lemma~\ref{l:cont-avg}\ref{l:cont-avg_cont}, $T_{\lambda,\mu}^\alpha$ is continuous and single-valued. 

\ref{p:gcoer-DR_reli}: Set $L:= \aff(A\cup B)$. Since $A$ and $B$ are convex with $\reli A\cap \reli B\neq\varnothing$, Lemma~\ref{l:cvxFix}\ref{l:cvxFix_reli} yields $\Fix T_{\lambda,\mu}^\alpha\cap L= A\cap B$. 
Now for every $w\in A\cap B$, we derive from \cite[Remark~8.2(v)]{BLPW13a} that $A$ is superregular at $w$, from \cite[Example~7.2(i)--(ii) and Proposition~7.5]{BLPW13a} that $\{A, B\}$ is affine-hull regular at $w$, and then from Corollary~\ref{c:coer-DR} that $T_{\lambda, \mu}^\alpha$ is $(A\cap B)$-quasi coercive on $\ball{w}{\delta}\cap L$ for some $\delta\in \RPP$. 
In turn, Lemma~\ref{l:global-coer} implies that $T_{\lambda, \mu}^\alpha$ is $(A\cap B)$-quasi coercive on $S\cap L$ for every bounded set $S$ of $X$.
Finally, apply Lemma~\ref{l:Lcoer} and note from Lemma~\ref{l:projL}\ref{l:projL_dist} that $\ball{0}{\delta}\subseteq P_L^{-1}(\ball{0}{\delta}\cap L)$. 

\ref{p:gcoer-DR_poly}: By \cite[Example~12.31(a)\&(d)]{RW98}, $P_A$ and $P_B$ are piecewise linear in the sense of \cite[Definition~2.47(a)]{RW98}. Notice that compositions of piecewise linear operators are also piecewise linear (see~\cite[Corollary~2.3]{Son82}), and that linear combinations of piecewise linear operators are also piecewise linear. Thus, $Q:= \Id-T_{\lambda, \mu}^\alpha=\alpha(\Id-P_B^\mu P_A^\lambda)$ is piecewise linear. By \cite[Example~2.48]{RW98}, $Q$ is also polyhedral (i.e., the graph of $Q$ is a union of finitely many polyhedral sets). Noting that $Q^{-1}(0)= \Fix T_{\lambda, \mu}^\alpha$ and $d_{Qx}(0)= \|Qx\|= \|x- T_{\lambda, \mu}^\alpha x\|$, and using \cite[Corollary of Proposition~1]{Rob81}, there exists $\nu\in \RPP$ and $\varepsilon\in \RPP$ such that 
\begin{equation} 
\|x- T_{\lambda, \mu}^\alpha x\|\geq \nu d_{\Fix T_{\lambda, \mu}^\alpha}(x)
\quad\text{whenever~} \|x- T_{\lambda, \mu}^\alpha x\|< \varepsilon.
\end{equation}
Now for every $w\in \Fix T_{\lambda, \mu}^\alpha$, since $T_{\lambda, \mu}^\alpha$ is continuous, there exists $\delta\in \RPP$ such that $\|x- T_{\lambda, \mu}^\alpha x\|< \varepsilon$ for all $x\in \ball{w}{\delta}$. It follows that $T_{\lambda, \mu}^\alpha$ is $(\Fix T_{\lambda, \mu}^\alpha, \nu)$-quasi coercive on $\ball{w}{\delta}$. Applying Lemma~\ref{l:global-coer} with $C= \Fix T_{\lambda, \mu}^\alpha$ and $L= X$ completes the proof.
\end{proof}

\subsection{In the presence of linear regularity}

Corollary~\ref{c:coer-DR} shows that superregularity and affine-hull regularity assumption is sufficient for quasi coercivity of gDR operators. We will show in Proposition~\ref{p:coer-DR-lin} that if $\min\{\lambda,\mu\}<2$, then affine-hull regularity can be replaced by linear regularity, a \emph{milder} assumption (see Remark~\ref{r:lin-aff}). This is a \emph{new} result that obtains quasi coercivity for gDR operators via linear regularity and operator parameters.

For $x,y,z\in X$, we denote $\widehat{xyz}$ the angle between two vectors $x-y$ and $z-y$, i.e., 
\begin{equation}
\widehat{xyz}\in \left[0, \pi\right]
\quad\text{such that}\quad
\cos\widehat{xyz}= \frac{\scal{x-y}{z-y}}{\|x-y\|\cdot\|z-y\|},
\end{equation}
with the convention that $\widehat{xyz}= 0$ if $x= y$ or $z= y$.
The following two lemmas are crucial for our analysis.

\begin{lemma}
\label{l:angle}
Let $(\varepsilon_n)_\nnn$ be a sequence in $\RP$ convergent to $0$. Let $(x_n)_\nnn$, $(r_n)_\nnn$, $(s_n)_\nnn$ and $(p_n)_\nnn$ be sequences in $X$ such that, for all $n\in \NN$, $r_n\notin \{x_n, s_n, p_n\}$, and that
\begin{equation}
\|x_n-s_n\|\leq \varepsilon_n(\|x_n-r_n\|+ \|s_n-r_n\|).
\end{equation}
Then $\widehat{x_nr_ns_n}\to 0$ and $\cos\widehat{x_nr_np_n}- \cos\widehat{s_nr_np_n}\to 0$ as $n\to +\infty$.
\end{lemma}
\begin{proof}
By assumption and Cauchy--Schwarz inequality, 
\begin{subequations}
\begin{align}
2\scal{x_n-r_n}{s_n-r_n}&=\|x_n-r_n\|^2+ \|s_n-r_n\|^2- \|x_n-s_n\|^2 \\
&\geq \|x_n-r_n\|^2+ \|s_n-r_n\|^2- \varepsilon_n^2(\|x_n-r_n\|+ \|s_n-r_n\|)^2 \\
&\geq (2-4\varepsilon_n^2)\|x_n-r_n\| \|s_n-r_n\|.
\end{align}
\end{subequations}
It follows that
\begin{equation}
\cos\widehat{x_nr_ns_n}= \frac{\scal{x_n-r_n}{s_n-r_n}}{\|x_n-r_n\| \|s_n-r_n\|}\geq 1-2\varepsilon_n^2\to 1
\quad\text{as~} n\to +\infty, 
\end{equation}
hence $\cos\widehat{x_nr_ns_n}\to 1$ and $\widehat{x_nr_ns_n}\to 0$ as $n\to +\infty$. 
Now we compute
\begin{equation}
\left\| \frac{x_n-r_n}{\|x_n-r_n\|}- \frac{s_n-r_n}{\|s_n-r_n\|} \right\|^2= 2- 2\frac{\scal{x_n-r_n}{s_n-r_n}}{\|x_n-r_n\| \|s_n-r_n\|}= 2(1- \cos\widehat{x_nr_ns_n})\to 0,
\end{equation}
and again by Cauchy--Schwarz inequality,
\begin{subequations}
\begin{align}
|\cos\widehat{x_nr_np_n}- \cos\widehat{s_nr_np_n}|&= \left|\scal{\frac{x_n-r_n}{\|x_n-r_n\|}- \frac{s_n-r_n}{\|s_n-r_n\|}}{\frac{p_n-r_n}{\|p_n-r_n\|}}\right| \\
&\leq \left\| \frac{x_n-r_n}{\|x_n-r_n\|}- \frac{s_n-r_n}{\|s_n-r_n\|} \right\|\cdot 1\to 0, 
\end{align}
\end{subequations}
which completes the proof.
\end{proof}

\begin{lemma}
\label{l:circumcircle}
Let $x, r, s$ and $p$ be in $X$ and set $u:= P_{L_1}p$ and $v:= P_{L_2}p$ with $L_1:= \aff\{x, r\}$ and $L_2:= \aff\{s, r\}$. Suppose that $r\notin \{x, s, u, v\}$. Then $\|u-v\|\leq \|p-r\|\sin\widehat{urv}= \|p-r\|\sin\widehat{xrs}$.
\end{lemma}
\begin{proof}
First, since $u\in \aff\{x, r\}$ and $v\in \aff\{s, r\}$, we have
\begin{equation}
|\cos\widehat{urv}|= \left| \scal{\frac{u-r}{\|u-r\|}}{\frac{v-r}{\|v-r\|}} \right| 
= \left| \scal{\frac{x-r}{\|x-r\|}}{\frac{s-r}{\|s-r\|}} \right|= |\cos\widehat{xrs}|, 
\end{equation}
which yields $\sin\widehat{urv}= \sin\widehat{xrs}$.

Set $q:= \frac{1}{2}(p+r)$. Then $\|q-r\|= \frac{1}{2}\|p-r\|$. 
Using Lemma~\ref{l:projL}\ref{l:projL_perp}, $\scal{p-u}{r-u}= 0$ and $\scal{p-v}{r-v}= 0$. 
We compute 
\begin{equation}
\|q-u\|^2= \tfrac{1}{4}\|(p-u)+ (r-u)\|^2= \tfrac{1}{4}\|(p-u)- (r-u)\|^2= \tfrac{1}{4}\|p-r\|^2
\end{equation}
and get $\|q-u\|= \frac{1}{2}\|p-r\|$. By the same argument for $\|q-v\|^2$, we obtain
\begin{equation}
\|q-r\|= \|q-u\|= \|q-v\|= \tfrac{1}{2}\|p-r\|.
\end{equation}
Now define $\hat{q}:= P_L q$ with $L:= \aff\{r, u, v\}$. By Lemma~\ref{l:projL}\ref{l:projL_ball},
\begin{equation}
\|q-r\|\geq \|\hat{q}-r\|= \|\hat{q}-u\|= \|\hat{q}-v\|,
\end{equation}
i.e., $\hat{q}$ is the center of the circumcircle passing $r$, $u$, and $v$. Applying the law of sines, we get
\begin{equation}
\|u-v\|= 2\|\hat{q}-r\|\sin\widehat{urv}\leq 2\|q-r\|\sin\widehat{urv}= \|p-r\|\sin\widehat{urv}. 
\end{equation}
The lemma is proved.
\end{proof}

We now ready to prove our \emph{new} result on quasi coercivity of gDR operators under linear regularity assumption.

\begin{proposition}[quasi coercivity of gDR operators under linear regularity]
\label{p:coer-DR-lin}
Let $A$ and $B$ be closed subsets of $X$ such that $A\cap B\neq \varnothing$. Let also $w\in A\cap B$, $\lambda,\mu\in \left]0, 2\right]$, and $\alpha\in \RPP$. Suppose that $\{A,B\}$ is superregular at $w$ and linearly regular around $w$ and that $\min\{\lambda, \mu\}< 2$. Then $T_{\lambda,\mu}^\alpha$ is $(A\cap B)$-quasi coercive around $w$.
\end{proposition}
\begin{proof}
By assumption, there exist $\kappa\in\RPP$ and $\delta\in \RPP$ such that 
\begin{equation}\label{e:170919c}
\forall x\in \ball{w}{\delta},\quad d_{A\cap B}(x)\leq \kappa\max\{d_A(x),d_B(x)\}.
\end{equation}
Now suppose on the contrary that $T_{\lambda,\mu}^\alpha$ is not $(A\cap B)$-quasi coercive around $w$. By Lemma~\ref{l:averaged-coer}, $P_B^\mu P_A^\lambda$ is not $(A\cap B)$-quasi coercive around $w$, which means that there exist sequences $\zeta_n\downarrow 0$, $x_n\to w$, $s_n\in P_B^\mu P_A^\lambda x_n$ such that
\begin{equation}\label{e:170919d}
0\leq \|x_n-s_n\|<\zeta_n d_{A\cap B}(x_n).
\end{equation}
We find $a_n\in P_A x_n$, $r_n\in P_A^\lambda x_n$, and $b_n\in P_B r_n$ such that
\begin{equation}
\label{e:relaxed}
x_n-r_n= \lambda(x_n-a_n) \quad\text{and}\quad s_n-r_n=\mu(b_n-r_n).
\end{equation}  
Without loss of generality, we assume that $\zeta_n< \min\{\frac{1}{\kappa}, \frac{\lambda}{\kappa}, \frac{\mu}{\kappa}\}$ and $x_n\in \ball{w}{\delta/(1+\lambda)}\subset \ball{w}{\delta}$ for all $\nnn$. 
It then follows from \eqref{e:170919c} and \eqref{e:170919d} that 
\begin{equation}
\forall\nnn,\quad \|x_n-s_n\|< \zeta_n d_{A\cap B}(x_n)\leq \zeta_n\kappa \max\{d_A(x_n),d_B(x_n)\}.
\end{equation}
Let $n\in \NN$. Since $a_n\in A$ and $b_n\in B$, we have $d_A(x_n)\leq \|x_n-a_n\|= \tfrac{1}{\lambda}\|x_n-r_n\|$ and 
\begin{equation}
d_B(x_n)\leq \|x_n-b_n\|\leq \|x_n-r_n\|+ \|b_n-r_n\|= \|x_n-r_n\|+ \tfrac{1}{\mu}\|s_n-r_n\|. 
\end{equation}
Therefore,
\begin{equation}\label{e:170919e}
\forall\nnn,\quad \|x_n-s_n\|< \varepsilon_n(\|x_n-r_n\|+\|s_n-r_n\|), 
\quad\text{where~} \varepsilon_n:= \zeta_n\kappa\max\{1,\tfrac{1}{\lambda},\tfrac{1}{\mu}\}\to 0.
\end{equation}
 Noting that $\|x_n-s_n\|\geq \|x_n-r_n\|- \|s_n-r_n\|$ and that $\varepsilon_n< 1$, we obtain 
\begin{equation}
\|x_n-r_n\|\leq
\tfrac{1+\varepsilon_n}{1-\varepsilon_n}\|s_n-r_n\|,
\end{equation}
which combined with \eqref{e:relaxed} yields
\begin{equation}
\label{e:170919h-b}
\forall\nnn,\quad
\|a_n-r_n\|= \tfrac{|\lambda-1|}{\lambda}\|x_n-r_n\|\leq \sigma_n\|b_n-r_n\|,
\quad\text{where~}
\sigma_n:=\tfrac{\mu|\lambda-1|}{\lambda}\cdot\tfrac{1+\varepsilon_n}{1-\varepsilon_n}.
\end{equation}

Next, let $p_n\in P_{A\cap B} r_n$. Since $x_n\to w$, Fact~\ref{f:ImBall}\ref{f:ImBall_P} implies that $a_n, r_n, b_n, s_n, p_n\to w$. In turn, $x_n-a_n\in \pnX{A}(a_n)$ and $r_n-b_n\in \pnX{B}(b_n)$. By the superregularity of $A$ and $B$ at $w$, taking subsequences if necessary and without relabeling, we can assume that
\begin{subequations}\label{e:170919p}
\begin{align}
\scal{x_n-a_n}{p_n-a_n}&\leq \varepsilon_n\|x_n-a_n\|\|p_n-a_n\|,\quad\text{and}
\label{e:170919p-a}\\
\scal{r_n-b_n}{p_n-b_n}&\leq \varepsilon_n\|r_n-b_n\|\|p_n-b_n\|.
\label{e:170919p-b}
\end{align}
\end{subequations}
Now we divide the proof into several parts.

{\bf\em Part 1:} We claim that
\begin{equation}\label{e:170919i}
\forall\nnn,\quad
x_n\notin \{r_n, a_n\}
\quad\text{and}\quad 
r_n\notin \{s_n, b_n, p_n\}
\end{equation}
and that 
\begin{equation}
\label{e:170919h-c}
\widehat{s_nr_nx_n}\to 0 \quad\text{and}\quad \cos\widehat{x_nr_np_n}- \cos\widehat{s_nr_np_n}\to 0 \quad\text{as~} n\to +\infty.
\end{equation}
Indeed, if $x_n=r_n$ or $x_n=a_n$ for some $n$, then $x_n=a_n=r_n\in A$ and, by \eqref{e:170919d}, 
\begin{subequations}
\begin{align}
\zeta_n d_{A\cap B}(x_n)&>\|x_n-s_n\|=\|r_n-s_n\|
=\mu\|r_n-b_n\|\\
&=\mu d_B(r_n)=\mu d_B(x_n)=\mu\max\{d_A(x_n),d_B(x_n)\}
\geq\tfrac{\mu}{\kappa}d_{A\cap B}(x_n),
\end{align}
\end{subequations}
which contradicts the fact that $\zeta_n< \frac{\mu}{\kappa}$. Similarly, if $r_n=s_n$ or $r_n=b_n$ for some $n$, then $r_n=b_n=s_n\in B$ and, by \eqref{e:170919d}, 
\begin{subequations}
\begin{align}
\zeta_n d_{A\cap B}(x_n)&>\|x_n-s_n\|=\|x_n-r_n\|
=\lambda\|x_n-a_n\|\\
&=\lambda d_A(x_n)=\lambda \max\{d_A(x_n),d_B(x_n)\}
\geq\tfrac{\lambda}{\kappa}d_{A\cap B}(x_n),
\end{align}
\end{subequations}
which contradicts the fact that $\zeta_n< \frac{\lambda}{\kappa}$. So we complete \eqref{e:170919i} due to $\|p_n-r_n\|\geq d_B(r_n)= \|b_n-r_n\|= \frac{1}{\mu}\|s_n-r_n\|$.
Now combining \eqref{e:170919e} and Lemma~\ref{l:angle}, we arrive at \eqref{e:170919h-c}.

{\bf\em Part~2:} Let $\overline{\sigma}_n:=\max\{1,\sigma_n\}$. Since $\varepsilon_n\to 0$, the sequence $(\sigma_n)_\nnn$ is bounded, and so is $(\overline{\sigma}_n)_\nnn$. We claim that
\begin{subequations}
\begin{align}
\forall\nnn,\quad 
&\|p_n-r_n\|\leq \overline{\sigma}_n\kappa\|b_n-r_n\|,
\label{e:170920c-a}\\
&\max\{\|p_n-a_n\|,\|p_n-b_n\|\}\leq \overline{\sigma}_n(\kappa+1)\|b_n-r_n\|.
\label{e:170920c-b}
\end{align}
\end{subequations}
Let $n\in \NN$. Since $x_n\in \ball{w}{\delta/(1+\lambda)}$ and $r_n\in P_A^\lambda x_n$, Fact~\ref{f:ImBall}\ref{f:ImBall_P} yields $r_n\in \ball{w}{\delta}$. Using linear regularity and \eqref{e:170919h-b}, we estimate
\begin{subequations}\label{e:170919u}
\begin{align}
\|p_n-r_n\|&=d_{A\cap B}(r_n)
\leq \kappa\max\{d_A(r_n),d_B(r_n)\}\\
&\leq \kappa\max\{\|a_n-r_n\|,\|b_n-r_n\|\}\\
&\leq \kappa\max\{\sigma_n\|b_n-r_n\|,\|b_n-r_n\|\}
=\kappa\overline{\sigma}_n\|b_n-r_n\|.
\end{align}
\end{subequations}
So \eqref{e:170920c-a} holds. Combining with \eqref{e:170919h-b} gives
\begin{subequations}
\begin{align}
\|p_n-a_n\|&\leq \|p_n-r_n\|+ \|a_n-r_n\|
\leq \big(\kappa\overline{\sigma}_n+\sigma_n\big)\|b_n-r_n\|,\\
\|p_n-b_n\|&\leq \|p_n-r_n\|+ \|b_n-r_n\|
\leq (\kappa\overline{\sigma}_n+1)\|b_n-r_n\|.
\end{align}
\end{subequations}
Thus, \eqref{e:170920c-b} holds.

{\bf\em Part 3:} We claim that $\lambda> 1$ and that there exists $\sigma\in \left]0, 1\right[$ satisfying
\begin{equation}\label{e:170927b}
\sigma_n\leq \sigma< 1
\quad\text{and}\quad
\overline{\sigma}_n=1
\quad\text{for all $n$ sufficiently large}.
\end{equation}
Let $\overline{\sigma}$ be an upper bound of the bounded sequence $(\overline{\sigma}_n)_\nnn$. 
Using \eqref{e:170919p-b} and \eqref{e:170920c-b}, we have
\begin{subequations}
\begin{align}
\cos\widehat{s_nr_np_n}&= \frac{\scal{s_n-r_n}{p_n-r_n}}{\|s_n-r_n\|\|p_n-r_n\|}= \frac{\scal{b_n-r_n}{p_n-r_n}}{\|b_n-r_n\|\|p_n-r_n\|} \\
&=\frac{\scal{b_n-r_n}{p_n-b_n} +\|b_n-r_n\|^2}{\|b_n-r_n\|\|p_n-r_n\|} \\
&\geq \frac{-\varepsilon_n \|b_n-r_n\|\|p_n-b_n\| +\|b_n-r_n\|^2}{\|b_n-r_n\|\|p_n-r_n\|} \\
&\geq \frac{1-\varepsilon_n \overline{\sigma}(\kappa+1)}{\kappa\overline{\sigma}}\to \frac{1}{\kappa\overline{\sigma}}> 0.
\end{align}
\end{subequations}
It follows that
\begin{equation}
\label{e:srp}
\cos\widehat{s_nr_np_n}> \frac{1}{2\kappa\overline{\sigma}}> 0
\quad\text{for all $n$ sufficiently large}.
\end{equation}
Combining with \eqref{e:170919h-c} yields
\begin{equation}
\label{e:xrp}
\cos\widehat{x_nr_np_n}> \frac{1}{4\kappa\overline{\sigma}}> 0
\quad\text{for all $n$ sufficiently large}.
\end{equation}
Suppose that $\lambda\leq 1$. 
Using \eqref{e:170919p-a}, \eqref{e:170920c-b} and noting that $\|p_n-r_n\|=d_{A\cap B}(r_n)\geq d_B(r_n)=\|r_n-b_n\|$, we obtain that
\begin{subequations}
\begin{align}
\cos\widehat{x_nr_np_n}&= \frac{\scal{x_n-r_n}{p_n-r_n}}{\|x_n-r_n\|\|p_n-r_n\|}= \frac{\scal{x_n-a_n}{p_n-r_n}}{\|x_n-a_n\|\|p_n-r_n\|} \\
&= \frac{\scal{x_n-a_n}{p_n-a_n}+ (\lambda-1)\|x_n-a_n\|^2}{\|x_n-a_n\|\|p_n-r_n\|} \\
&\leq \varepsilon_n\frac{\|p_n-a_n\|}{\|p_n-r_n\|}+ (\lambda-1)\frac{\|x_n-a_n\|}{\|p_n-r_n\|} 
\leq \varepsilon_n\overline{\sigma}(\kappa+1)\to 0,
\end{align}
\end{subequations}
which contradicts \eqref{e:xrp}. We must therefore have $\lambda>1$.
Now notice that $\frac{|\lambda-1|}{\lambda}=\frac{\lambda-1}{\lambda}\leq\frac{1}{2}$ and $\mu\leq 2$, where only one equality can happen. Thus,
\begin{equation}
\sigma_n= \tfrac{\mu(\lambda-1)}{\lambda}\cdot \tfrac{1+\varepsilon_n}{1-\varepsilon_n}\to \tfrac{\mu(\lambda-1)}{\lambda}< 1.
\end{equation}
The claim then follows.

{\bf\em Part 4:} Define lines $L_{1,n}:= \aff\{x_n, r_n\}$, $L_{2,n}:= \aff\{r_n, s_n\}$ and projections $u_n:= P_{L_{1,n}}p_n$, $v_n:= P_{L_{2,n}}p_n$. We have that 
\begin{equation}\label{e:170919s}
u_n:=a_n+\eta_1\frac{x_n-a_n}{\|x_n-a_n\|}
\quad\text{with}\quad
\eta_1:=\frac{\scal{x_n-a_n}{p_n-a_n}}{\|x_n-a_n\|}
\leq \varepsilon_n\|p_n-a_n\|
\end{equation}
and that 
\begin{equation}\label{e:170919r}
v_n:=b_n+\eta_2\frac{r_n-b_n}{\|r_n-b_n\|},
\quad\text{with}\quad
\eta_2:=\frac{\scal{r_n-b_n}{p_n-b_n}}{\|r_n-b_n\|}
\leq \varepsilon_n\|p_n-b_n\|,
\end{equation}
where the upper bound for $\eta_1$ and $\eta_2$ follows from \eqref{e:170919p}. 
By using \eqref{e:170919h-b}, \eqref{e:170920c-b}, and \eqref{e:170927b}, for all $n$ sufficiently large,
\begin{subequations}
\begin{align}
\|u_n-r_n\|&\leq \|a_n-r_n\|+ \eta_1\leq \|a_n-r_n\|+ \varepsilon_n\|p_n-a_n\|
\leq \big(\sigma+ \varepsilon_n(\kappa+1)\big)\|b_n-r_n\|, \\
\|v_n-r_n\|&= |\|b_n-r_n\|-\eta_2|
\geq \|b_n-r_n\|-\varepsilon_n\|p_n-b_n\|
\geq\big(1- \varepsilon_n(\kappa+1)\big)\|b_n-r_n\|,
\end{align}
\end{subequations}
and so 
\begin{equation}
\label{e:vnun}
\|u_n-v_n\|\geq \|v_n-r_n\|-\|u_n-r_n\|\geq \tfrac{1-\sigma}{2}\|b_n-r_n\|,
\end{equation}
which together with \eqref{e:170919i}, \eqref{e:170920c-a}, and \eqref{e:170927b} yields
\begin{equation}
\|u_n-v_n\|\geq \tfrac{1-\sigma}{2\kappa}\|p_n-r_n\|> 0.
\end{equation}
On the other hand, for all $n$ sufficiently large, $r_n\notin \{u_n, v_n\}$ due to \eqref{e:srp} and \eqref{e:xrp}. Noting also from \eqref{e:170919i} that $r_n\notin \{x_n, s_n\}$, we then apply Lemma~\ref{l:circumcircle} to get $\|u_n-v_n\|\leq \|p_n-r_n\|\sin\widehat{x_nr_ns_n}$, hence
\begin{equation}
0< \tfrac{1-\sigma}{2\kappa}\|p_n-r_n\|\leq \|p_n-r_n\|\sin\widehat{x_nr_ns_n}.
\end{equation}
Using \eqref{e:170919h-c}, we derive that
\begin{equation}
0<\tfrac{1-\sigma}{2\kappa}\leq \sin\widehat{x_nr_ns_n}\to 0,
\end{equation}
which is a contradiction.
\end{proof}

\begin{remark}
In Proposition~\ref{p:coer-DR-lin}, the parameter condition $\min\{\lambda,\mu\}<2$ cannot be removed. For example, we consider two convex (hence, superregular) sets $A:=\epi|\cdot|=\menge{(s,t)\in\RR^2}{s\leq|t|}$ and $B:=\RR\times\{0\}$ in $\RR^2$. Clearly, $\{A,B\}$ is linearly regular at $(0,0)\in A\cap B$. Consider the DR operator $T:=T_{2,2}^\alpha$ for some $\alpha\in\left]0,1\right]$ and $x=(0,t)$ for $t<0$. In this case $\lambda=\mu=2$, and we check that $x=Tx$, while $d_{A\cap B}(x)= |t|$. Therefore, $T$ is not $(A\cap B)$-quasi coercive at $(0,0)$.
\end{remark}

We also obtain a global version of Proposition~\ref{p:coer-DR-lin}.
\begin{proposition}[global quasi coercivity of convex gDR operators]
\label{p:gcoer-DR-lin}
Let $A$ and $B$ be closed convex subsets of $X$ such that $A\cap B\neq \varnothing$ and that $\{A,B\}$ is boundedly linearly regular. Let $\lambda,\mu\in \left]0,2\right]$ be such that $\min\{\lambda,\mu\}< 2$, and let $\alpha\in \RPP$. Then $T_{\lambda,\mu}^\alpha$ is $(A\cap B)$-quasi coercive on every bounded set $S$ of $X$.
\end{proposition}
\begin{proof}
On the one hand, by assumption, $T_{\lambda,\mu}^\alpha$ is a continuous single-valued operator and $\Fix T_{\lambda,\mu}^\alpha= A\cap B$ due to Lemma~\ref{l:cont-avg}\ref{l:cont-avg_cont} and Lemma~\ref{l:cvxFix}\ref{l:cvxFix_intersect}. 
On the other hand, for every $w\in A\cap B$, $\{A, B\}$ is superregular at $w$ (see \cite[Remark~8.2(v)]{BLPW13a}) and linearly regular around $w$, hence, by Proposition~\ref{p:coer-DR-lin}, $T_{\lambda,\mu}^\alpha$ is $(A\cap B)$-quasi coercive around $w$.  
Applying Lemma~\ref{l:global-coer} with $L= X$, we obtain that $T_{\lambda,\mu}^\alpha$ is $(A\cap B)$-quasi coercive on every bounded set $S$ of $X$.
\end{proof}

As a supplement for the above result, we refer to \cite[Corollary~5]{BBL99} for the most common \emph{sufficient condition} that guarantees bounded linear regularity for finite systems of convex sets.

\section{Convergence rate analysis}
\label{s:convg}

In this section, let $m$ be a positive integer, set $I:= \{1, \dots, m\}$, and let $\{C_i\}_{i\in I}$ be a system of closed (possibly nonconvex) subsets of $X$ with $C:=\bigcap_{i\in I}C_i\neq \varnothing$. 
Given an ordered tuple $(T_i)_{i\in I}$ of set-valued operators from $X$ to $X$, the \emph{cyclic algorithm} associated with $(T_i)_{i\in I}$ generates \emph{cyclic sequences} $(x_n)_\nnn$ by
\begin{equation}
\label{e:cycseq}
\forall\nnn,\quad x_{n+1}\in T_{n+1} x_n, \quad\text{where~} x_0 \in X.
\end{equation}
Here we adopt the convention that
\begin{equation}
\label{e:cvn}
\forall\nnn,\ \forall i\in I,\quad T_{mn+i}:= T_i.
\end{equation}

Recall that a sequence $(x_n)_\nnn$ is said to converge to a point $\ox$ with \emph{$R$-linear rate} $\rho\in\left[0,1\right[$ if there exists a constant $\sigma\in\RP$ such that
\begin{equation}
\forall\nnn,\quad\|x_n-\ox\|\leq\sigma\rho^n.
\end{equation}
In what follows, we denote $[\rho]_+ :=\max\{0, \rho\}$ for $\rho\in \RR$.
\begin{theorem}[sufficient condition for linear convergence]
\label{t:lincvg}
Let $w\in C$, $\delta\in \RPP$, and $\nu\in \left]0, 1\right]$. For every $i\in I$, let $\gamma_i\in \left[1, +\infty\right[$ and $\beta_i\in \RPP$. Let $(x_n)_\nnn$ be a cyclic sequence generated by $(T_i)_{i\in I}$. Suppose that
\begin{enumerate}[label={\rm(\alph*)}]
\item $\{C_i\}_{i\in I}$ is $\kappa$-linearly regular on $\ball{w}{\delta/2}$ for some $\kappa \in \RPP$.
\item For every $i \in I$, $T_i$ is $(C_i\cap \ball{w}{\delta}, \gamma_i, \beta_i)$-quasi firmly Fej\'er monotone and $(C_i, \nu)$-quasi coercive on $\ball{w}{\delta/2}$.
\end{enumerate}
Set $\Gamma :=(\gamma_1\cdots \gamma_m)^{1/2}$, $\delta_0 :=\frac{\delta}{2\Gamma}\gamma_m^{1/2}$, and 
\begin{equation}
\label{e:rate}
\rho:= \bigg[\Gamma^2-\frac{\nu^2}{\kappa^2}\Big(\sum_{i\in I}\frac{1}{\beta_i}\Big)^{-1}\bigg]_+^{1/2}.
\end{equation}
Then the following hold:
\begin{enumerate}
\item\label{t:lincvg_ine}
$\forall x_0\in\ball{w}{\delta_0},\quad d_C(x_m)\leq \rho d_C(x_0)$. 
\item\label{t:lincvg_local}
If $\rho<1$, then whenever $(x_{mn})_\nnn\subset \ball{w}{\delta_0}$ or $x_0\in\bball{w}{\frac{\delta_0(1-\rho)}{2+\Gamma-\rho}}$, the sequence $(x_n)_\nnn$ converges $R$-linearly to a point $\ox\in C$ with rate $\rho^{1/m}$.
\item\label{t:lincvg_global}
If $\gamma_i= 1$ for every $i\in I$, then whenever $x_0\in \ball{w}{\delta/2}$, the sequence $(x_n)_\nnn$ converges $R$-linearly to a point $\ox\in C$ with rate  
\begin{equation}
\bigg[1-\frac{\nu^2}{\kappa^2}\Big(\sum_{i\in I}\frac{1}{\beta_i}\Big)^{-1}\bigg]_+^{1/2m}.
\end{equation}
\end{enumerate}
\end{theorem}
\begin{proof}
\ref{t:lincvg_ine}\&\ref{t:lincvg_local}: This follows from \cite[Theorem~4.5]{DP16}.

\ref{t:lincvg_global}: Assume that $\gamma_i= 1$ for every $i\in I$. Then $\Gamma= 1$, $\delta_0= \delta/2$, and
\begin{equation}
\rho= \bigg[1-\frac{\nu^2}{\kappa^2}\Big(\sum_{i\in I}\frac{1}{\beta_i}\Big)^{-1}\bigg]_+^{1/2m}< 1.
\end{equation} 
Let $x_0\in \ball{w}{\delta/2}$. Since for every $i\in I$, $T_i$ is $(C_i\cap \ball{w}{\delta}, 1, \beta_i)$-quasi firmly Fej\'er monotone and hence $(C_i\cap \ball{w}{\delta}, 1)$-quasi Fej\'er monotone on $\ball{w}{\delta/2}$, we obtain from \cite[Lemma~3.4(ii)]{DP16} that $(x_n)_\nnn\subset \ball{w}{\delta/2}$. Now apply \ref{t:lincvg_local}.
\end{proof}

From now on, let $\ell$ be a positive integer and set $J:= \{1, \dots, \ell\}$. For every $j \in J$, let
\begin{subequations}\label{e:parameters}	
\begin{align}
&\lambda_j, \mu_j\in \left]0, 2\right],\ 
\hbeta_j :=\Big(\frac{\lambda_j}{2-\lambda_j} +\frac{\mu_j}{2-\mu_j}\Big)^{-1},\ 
\alpha_j\in \big]0, 1+\hbeta_j\big[, \text{~and} 
\label{e:parameters_lm&mu}\\
&s_j, t_j\in I\quad\text{such that~} s_j\neq t_j \text{~and~}
\{s_j\}_{j \in J}\cup \{t_j\}_{j\in J} =I.
\label{e:parameters_I&J}
\end{align}
\end{subequations}
Setting 
\begin{equation}
\label{e:gDR}
\forall j \in J,\quad T_j :=(1 -\alpha_j)\Id +\alpha_j P_{C_{t_j}}^{\mu_j}P_{C_{s_j}}^{\lambda_j},
\end{equation}
we study the \emph{cyclic generalized Douglas--Rachford algorithm} defined by $(T_j)_{j\in J}$, which includes several algorithms in the literature, for example, the \emph{cyclically anchored DR algorithm} \cite{BNP15} and the \emph{cyclic DR algorithm} \cite{BT14}; see \cite[Section~5.3]{DP16} for more details. We also say that the cyclic gDR algorithm is \emph{connected} if for every $i, k\in I$, there exists a path
\begin{equation}
\label{e:connected}
\{(i_1,i_2),(i_2,i_3),\ldots,(i_{q-1},i_{q})\}
\subseteq \{(s_j,t_j),(t_j,s_j)\}_{j\in J}
\quad\text{such that $i_1=i$, $i_q =k$};
\end{equation}
in other words, $I=\{1,\ldots,m\}$ and $\{(s_j,t_j)\}_{j\in J}$ respectively represent the vertices and edges of a \emph{connected undirected graph}. Here, a graph is undirected if every edge is bidirectional, and is connected if every two vertices can be linked by some path composed by the edges. 
It is clear that the cyclically anchored DR algorithm and the cyclic DR algorithm are connected.

Next, for every $j\in J$, we define
\begin{equation}
\label{e:LjZj}
L_j :=\aff(C_{s_j}\cup C_{t_j}) \quad\text{and}\quad
Z_j :=\begin{cases}
(C_{s_j}\cap C_{t_j})+ (L_j-L_j)^\perp &\text{~if~} \lambda_j =\mu_j =2, \\
C_{s_j}\cap C_{t_j} &\text{~otherwise}
\end{cases}
\end{equation}
and note from Lemma~\ref{l:Fix}\ref{l:Fix_inclusion} that $Z_j\subseteq\Fix T_j$. A relationship between $\{Z_j\}_{j\in J}$ and $\{C_i\}_{i\in I}$ is given as follows.

\begin{lemma}[$\{Z_j\}_{j\in J}$ vs. $\{C_i\}_{i\in I}$]
\label{l:ZjCi}
Let $w\in\bigcap_{i\in I}C_i$. Suppose that $\{C_{s_j},C_{t_j}\}$ is strongly regular at $w$ whenever $\lambda_j=\mu_j=2$. Then
\begin{equation}
\bigcap_{j\in J}Z_j=\bigcap_{i\in I}C_i.
\end{equation}
If, in addition, $\{C_i\}_{i\in I}$ is $\kappa$-linearly regular around $w$, then so is $\{Z_j\}_{j\in J}$.
\end{lemma}
\begin{proof}
By assumption and Proposition~\ref{p:Lstr}, $L_j =X$ whenever $\lambda_j =\mu_j =2$. Thus, \eqref{e:LjZj} implies that $Z_j =C_{s_j}\cap C_{t_j}$ for every $i\in J$. So it follows from \eqref{e:parameters_I&J} that
\begin{equation}\label{e:ZjCi}
\bigcap_{j\in J}Z_j =\bigcap_{j\in J}(C_{s_j}\cap C_{t_j}) =\bigcap_{i\in I}C_i.
\end{equation} 
Next, we note that
\begin{equation}
\forall j\in J,\ \forall x \in X,\quad \max\{d_{C_{s_j}}(x), d_{C_{t_j}}(x)\} \leq d_{Z_j}(x).
\end{equation}
Taking the maximum over all $j\in J$ and using \eqref{e:parameters_I&J} yield
\begin{equation}\label{e:maxC<maxZ}
\forall x \in X,\quad \max_{i\in I} d_{C_i}(x) \leq \max_{j\in J} d_{Z_j}(x).  
\end{equation}
Now suppose in addition that $\{C_i\}_{i\in I}$ is $\kappa$-linearly regular around $w$.
Then combining with \eqref{e:ZjCi} and \eqref{e:maxC<maxZ}, we deduce that $\{Z_j\}_{j\in J}$ is also $\kappa$-linearly regular around $w$.
\end{proof}

\begin{lemma}[shadows of common fixed points]
\label{l:cfixedpoints}
Suppose that the cyclic gDR algorithm is connected. Then
\begin{equation}
\forall\ox\in \bigcap_{j\in J} Z_j,\quad P_{C_1}\ox= \cdots= P_{C_m}\ox= P_C\ox\in C.
\end{equation}
\end{lemma}
\begin{proof}
Let $\ox\in \bigcap_{j\in J} Z_j$. By Lemma~\ref{l:Fix}\ref{l:Fix_projection}, $P_{C_{s_j}}\ox= P_{C_{t_j}}\ox\in C_{s_j}\cap C_{t_j}$ for every $j\in J$. Since the algorithm is connected, in view of \eqref{e:parameters_I&J} and \eqref{e:connected}, we conclude that $P_{C_1}\ox= \cdots= P_{C_m}\ox\in C$, which also implies that $P_{C_i}\ox= P_C\ox$ for every $i\in I$.
\end{proof}

We note from Propositions~\ref{p:Lstr} and \ref{p:str&CQ} that $\{C_{s_j}, C_{t_j}\}$ is affine-hull regular at $w$ if and only if the CQ-number $\theta_{C_{s_j},C_{t_j},L_j}(w,\delta)<1$ for some $\delta\in\RPP$, in which case $\{C_{s_j}, C_{t_j}\}$ is linearly regular around $w$. This perspective supports the use of our assumptions in the following.

\begin{theorem}[linear convergence under affine-hull regularity]
\label{t:cDR}
Let $w\in C:=\bigcap_{i\in I}C_i$, $\varepsilon\in \left[0, 1/3\right]$, $\delta\in \RPP$, $\kappa\in \RPP$, and $\kappa_j\in \RPP$ for every $j\in J$. Suppose that  
\begin{enumerate}[label={\rm(\alph*)}]
\item 
\label{t:cDR_lin}
$\{Z_j\}_{j\in J}$ is $\kappa$-linearly regular on $\ball{w}{\delta/2}$ and for every $j\in J$, $\{C_{s_j}, C_{t_j}\}$ is $\kappa_j$-linearly regular on $\ball{w}{\delta/2}$.	
\item 
\label{t:cDR_reg}
For every $i\in I$, $C_i$ is $(\varepsilon, \sqrt{2}\delta)$-regular at $w$.
\item 
\label{t:cDR_CQ}
For every $j\in J$, the CQ-number $\theta_j :=\theta_{C_{s_j},C_{t_j},L_j}(w,\sqrt{2}\delta)< 1$.
\item 
\label{t:cDR_rate}
Setting for every $j\in J$, 
\begin{subequations}
\begin{align}
\gamma_j&:=1 -\alpha_j +\alpha_j\big(1 +\tfrac{\lambda_j\varepsilon}{1 -\varepsilon}\big)\big(1 +\tfrac{\mu_j\varepsilon}{1 -\varepsilon}\big),\\
\nu_j&:=\tfrac{\alpha_j\sqrt{1 -\theta_j^2}}{\kappa}\min\Big\{\lambda_j, \tfrac{2\mu_j}{|1-\mu_j|+\sqrt{(1-\mu_j)^2+4\mu_j(1-\theta_j^2)}}\Big\},\quad\text{and}\\
\nu'_j &:=\begin{cases}
\nu_j &\text{~if~} \lambda_j =\mu_j =2 \text{~or~} L_j=X, \\
\min\{\nu_j, \alpha_j(\lambda_j +\mu_j -\lambda_j\mu_j)\} &\text{~otherwise},
\end{cases}
\end{align}
\end{subequations}
it holds that
\begin{equation}
\rho :=\left[\Gamma^2 -\frac{\nu^2}{\kappa^2} \Big(\sum_{j\in J} \frac{\alpha_j}{1-\alpha_j+\hbeta_j}\Big)^{-1}\right]_+^{\frac{1}{2\ell}}< 1,
\end{equation}
where $\Gamma :=(\gamma_1\cdots \gamma_\ell)^{1/2}$ 
and $\nu :=\min\limits_{j\in J} \{\nu'_j, 1\}$.
\end{enumerate}	
Then if either $(x_{\ell n})_\nnn\subset\ball{w}{\delta_0}$ or $x_0\in\ball{w}{\frac{\delta_0(1-\rho)}{2+\Gamma-\rho}}$, where $\delta_0 :=\frac{\delta}{2\Gamma}\gamma_{\ell}^{1/2}$, the cyclic sequence $(x_n)_\nnn$ generated by $(T_j)_{j\in J}$ converges $R$-linearly with rate $\rho$ to a point 
\begin{equation}\label{e:cDR-1}
\ox\in \bigcap_{j\in J} Z_j\subseteq \bigcap_{j\in J} \Fix T_j.
\end{equation} 
Additionally, if the cyclic gDR algorithm is connected, then $P_{C_1}\ox =\cdots =P_{C_m}\ox\in C$.
\end{theorem}
\begin{proof} 
Let $j \in J$. We have that $w\in C_{s_j}\cap C_{t_j}$ and, by Proposition~\ref{p:coer-DR}, that $T_j$ is $(Z_j, \nu'_j)$- and therefore $(Z_j, \nu)$-quasi coercive on $\ball{w}{\delta/2}$. 
Noting that 
\begin{equation}
\big(C_{s_j}\cap C_{t_j}+ (L_j-L_j)^\perp\big)\cap \ball{w}{\delta}\subseteq \big(C_{s_j}\cap C_{t_j}\cap \ball{w}{\delta}\big)+ (L_j-L_j)^\perp
\end{equation}
and using Proposition~\ref{p:qfF-DR}\ref{p:qfF-DR_X}--\ref{p:qfF-DR_22}, we derive that $T_j$ is $(Z_j\cap\ball{w}{\delta}, \gamma_j, \frac{1-\alpha_j+\hbeta_j}{\alpha_j})$-quasi firmly Fej\'er monotone on $\ball{w}{\delta/2}$. 
Thus, applying Theorem~\ref{t:lincvg}\ref{t:lincvg_local} to $(T_j)_{j\in J}$ and the corresponding sets $(Z_j)_{j\in J}$, we obtain $R$-linear convergence of the cyclic sequence $(x_n)_\nnn$ to a point $\ox$ satisfying \eqref{e:cDR-1}.
Now Lemma~\ref{l:cfixedpoints} completes the proof.
\end{proof}

\begin{remark}[sharper convergence rate]
\label{r:better_rate}
Theorem~\ref{t:cDR} indeed generalizes \cite[Theorem~5.21]{DP16}, and moreover, provides sharper convergence rate under the \emph{same} assumptions. To see this, let $w\in\bigcap_{i\in I}C_i$ and suppose all assumptions \cite[Theorem~5.21]{DP16} are fulfilled. By Proposition~\ref{p:Lstr}, Proposition~\ref{p:str&CQ}, and Lemma~\ref{l:ZjCi}, all assumptions in Theorem~\ref{t:cDR} are also satisfied on some neighborhood of $w$. It can be seen that the linear convergence rate $\rho$ in Theorem~\ref{t:cDR} is smaller than the one obtained in the proof of \cite[Theorem~5.21]{DP16} because its corresponding quasi coercivity constant $\nu$ is better (see Remark~\ref{r:imprv_qcoer}).
\end{remark}

\begin{remark}[shadows of common fixed points]
As shown in Theorem~\ref{t:cDR}, the cyclic gDR algorithm converges (locally) to the set of common fixed points. However, without \emph{additional} conditions, one shoud not expect the limit points or their projections (or ``shadows'') onto $C_i$'s to lie in the intersection $C :=\bigcap_{i\in I}C_i$, which means that those points might not solve the feasibility problem! We will illustrate this phenomenon in the following example.

Suppose that $C_1 =\RP\times\RR\times\{0\}$, $C_2 =\RR\times\RP\times\{0\}$, $C_3 =\menge{(\xi, \zeta ,\zeta)}{\xi, \zeta\in \RR}$, and $C_4 =\RP\times\RR^2$ in $X =\RR^3$. Consider $J =\{1,2\}$, $\lambda_1 =\mu_1 =2$, $\min\{\lambda_2,\mu_2\} <2$, $(s_1,t_1) =(1,2)$, and $(s_2,t_2) =(3,4)$. So $T_1:=T_{2,2}^{\alpha_1}$ and $T_2:=T_{\lambda_2,\mu_2}^{\alpha_2}$ are the gDR operators for $(C_1,C_2)$ and $(C_3,C_4)$, respectively. Then 
\begin{subequations}
\begin{align}
&L_1=\aff(C_1\cup C_2)=\RR^2\times\{0\},
&&\ Z_1=(C_1\cap C_2)+(L_1-L_1)^\perp=\RP^2\times\RR,\\
&L_2=\aff(C_3\cup C_4)=\RR^3,
&&\ Z_2=(C_3\cap C_4)=\menge{(\xi, \zeta, \zeta)}{\xi, \zeta\in \RP}.
\end{align}
\end{subequations}
Now take $\ox=(1,1,1)\in Z_1\cap Z_2\subseteq\Fix T_1\cap\Fix T_2$. Then $P_{C_1}\ox= P_{C_2}\ox=(1,1,0)\neq (1,1,1)=P_{C_3}\ox=P_{C_4}\ox$.
So these projections are not identical and neither of them lies in the intersection $C_1\cap C_2\cap C_3\cap C_4=\RP\times\{0\}\times\{0\}$.
\end{remark}

\begin{remark}[on linear regularity moduli]
To the best of our knowledge, there is no known results on the relationship between the linear regularity modulus $\kappa$ of the entire system $\{C_i\}_{i\in I}$ and those of its subsystems. So we present here two simple examples showing that one modulus can be arbitrarily large while others remain bounded.

We will need the following formula whose proof is elementary: For two intersecting hyperplanes $A$ and $B$ of $X$ with two nonparallel unit normal vectors $n_A$ and $n_B$, the system $\{A,B\}$ is linearly regular on $X$ with modulus
\begin{equation}\label{e:kappa}
\sqrt{\frac{2}{1-|\scal{n_A}{n_B}|}}.
\end{equation}
\begin{enumerate}
\item {\bf Arbitrarily large linear regularity modulus for subsystem}:
Let $\varepsilon\in\RPP$ and suppose that $C_1 =\RR\times\{0\}$, $C_2 =\menge{\xi(1, \varepsilon)}{\xi\in \RR}$, $C_3 =\{0\}\times\RR$ are lines in $X =\RR^2$. One can check that $\bigcap_{i=1}^3 C_i =\{(0,0)\}$ and that $\{C_i\}_{i\in\{1,2,3\}}$ is $\kappa$-linearly regular on $X$ with $\kappa=\sqrt{2}$. As noticed, $\{C_1,C_2\}$ is linearly regular on $X$. Let $\kappa'$ be a linear regularity modulus of $\{C_1,C_2\}$ around $w =(0, 0)$ and take $x =(\varepsilon, \varepsilon^2)\in C_2$. Then, as $\varepsilon$ is sufficiently small, 
$\sqrt{\varepsilon^2 +\varepsilon^4} =d_{C_1\cap C_2}(x)\leq \kappa'\max\{d_{C_1}(x),d_{C_2}(x)\} =\kappa'\varepsilon^2$. We deduce that $\kappa'\geq \sqrt{1+1/\varepsilon^2}$ and so $\kappa'$ can be arbitrarily large while $\kappa$ remains constant.

\item {\bf Arbitrarily large linear regularity modulus for entire system}:
Let $\varepsilon\in\RPP$ and suppose that $X =\RR^3$. Consider the planes $C_1 =\RR^2\times\{0\}$, $C_2 =\{0\}\times\RR^2$, $C_3$ being the plane defined by $\{(0,0,0),(\varepsilon,1,0),(0,1,\varepsilon)\}$, and let $w =(0,0,0)\in C:=C_1\cap C_2\cap C_3=\{(0,0,0)\}$. We see that $C_1$, $C_2$, and $C_3$ respectively have unit normal vectors 
\begin{equation}
n_1=(0,0,1),\
n_2=(1,0,0),\text{~and~}
n_3=\tfrac{1}{\sqrt{2+\varepsilon^2}}(1,-\varepsilon,1).
\end{equation}
So $\{C_i,C_j\}$ is $\kappa_{i,j}$-linearly regular on $X$, where $\kappa_{i,j}$ is computed by \eqref{e:kappa} as
$\kappa_{1,2}=\sqrt{2}$ and
$\kappa_{1,3}=\kappa_{2,3}
=\sqrt{2}/\sqrt{1-1/\sqrt{2+\varepsilon^2}}\in\left]\sqrt{2},2\right[$. Now assume that $\{C_1,C_2,C_3\}$ is $\kappa$-linearly regular around $w$ for some $\kappa\in\RP$ and let $x =(\varepsilon^2, \varepsilon, 0)\in C_1\cap C_3$. Then, as $\varepsilon$ is sufficiently small,
\begin{equation}
\sqrt{\varepsilon^2+\varepsilon^4} =d_C(x)\leq \kappa\max_{i\in\{1,2,3\}}d_{C_i}(x) =\kappa\varepsilon^2,
\end{equation}
which yields $\kappa\geq \sqrt{1+1/\varepsilon^2}$. Hence, $\kappa$ can be arbitrarily large while $\kappa_{i,j}$ remains bounded.
\end{enumerate}
\end{remark}

We note from Proposition~\ref{p:str&CQ}\ref{p:str&CQ_str}--\ref{p:str&CQ_CQ} that assumption~\ref{t:cDR_CQ} in Theorem~\ref{t:cDR} means that for every $j\in J$, $\{C_{s_j}, C_{t_j}\}$ is affine-hull regular ar $w$. Nevertheless, in the following linear convergence result, we only require linear regularity for pairs $\{C_{s_j}, C_{t_j}\}$ corresponding to $\min\{\lambda_j, \mu_j\}< 2$.   
\begin{theorem}[linear convergence under linear regularity]
\label{t:super}
Let $w\in C:=\bigcap_{i\in I}C_i$. 
Suppose that $\{Z_j\}_{j\in J}$ is linearly regular around $w$, that $\{C_i\}_{i\in I}$ is superregular at $w$, and that for every $j\in J$, $\{C_{s_j}, C_{t_j}\}$ is linearly regular around $w$ if $\min\{\lambda_j, \mu_j\}< 2$ and affine-hull regular at $w$ otherwise.
Then the cyclic gDR algorithm converges $R$-linearly locally to a point 
\begin{equation}
\ox\in \bigcap_{j\in J} Z_j\subseteq \bigcap_{j\in J} \Fix T_j.
\end{equation} 
Moreover, $P_{C_1}\ox =\cdots =P_{C_m}\ox\in C$ provided that the cyclic gDR algorithm is connected.
\end{theorem}
\begin{proof}
Combining Corollary~\ref{c:coer-DR}\ref{c:coer-DR_gen} and Proposition~\ref{p:coer-DR-lin}, there exist $\nu\in \left]0, 1\right]$ and $\delta\in \RPP$ such that, for every $j\in J$, $T_j$ is $(Z_j, \nu)$-quasi coercive on $\ball{w}{\delta/2}$. 
Let $\varepsilon \in \left]0, 1/3\right]$. Since $\{C_i\}_{i\in I}$ is superregular at $w$, we shrink $\delta$ if necessary so that $C_i$ is $(\varepsilon, \sqrt{2}\delta)$-regular at $w$ for every $i\in I$. Now let $j\in J$. Then $C_{s_j}$ and $C_{t_j}$ are respectively $(\varepsilon, \delta)$- and $(\varepsilon, \sqrt{2}\delta)$-regular at $w$. Using Proposition~\ref{p:qfF-DR}\ref{p:qfF-DR_X}--\ref{p:qfF-DR_22} and noting that 
\begin{equation}
\big(C_{s_j}\cap C_{t_j}+ (L_j-L_j)^\perp\big)\cap \ball{w}{\delta}\subseteq \big(C_{s_j}\cap C_{t_j}\cap \ball{w}{\delta}\big)+ (L_j-L_j)^\perp,
\end{equation}
we have that $T_j$ is $(Z_j\cap\ball{w}{\delta}, \gamma_j, \frac{1-\alpha_j+\hbeta_j}{\alpha_j})$-quasi firmly Fej\'er monotone on $\ball{w}{\delta/2}$, where $\gamma_j:=1 -\alpha_j +\alpha_j\big(1 +\tfrac{\lambda_j\varepsilon}{1 -\varepsilon}\big)\big(1 +\tfrac{\mu_j\varepsilon}{1 -\varepsilon}\big)$. 
Since $\gamma_j\to 1^+$ as $\varepsilon\to 0^+$, we can choose $\varepsilon$ sufficiently small so that
\begin{equation}
\rho :=\left[\gamma_1 \cdots \gamma_m -\frac{\nu^2}{\kappa^2}\big(\sum_{j\in J} \frac{\alpha_j}{1 -\alpha_j+\hbeta_j}\big)^{-1}\right]_+^{1/2}<1.
\end{equation}
Now $R$-linear convergence of the cyclic sequence $(x_n)_\nnn$ is obtained by applying Theorem~\ref{t:lincvg}\ref{t:lincvg_local} to $(T_j)_{j\in J}$ and the corresponding sets $(Z_j)_{j\in J}$. The rest then follows from Lemma~\ref{l:cfixedpoints}.
\end{proof}

In Theorem~\ref{t:super}, if $\{C_{s_j}, C_{t_j}\}$ is strongly regular instead of affine-hull regular at $w$ whenever $\lambda_j =\mu_j =2$, then the limit point $\ox\in C$. In this case, by Lemma~\ref{l:ZjCi}, the linear regularity of $\{Z_j\}_{j\in J}$ is a consequence of that of $\{C_i\}_{i\in I}$. We summarize this observation in the following corollary, which indeed extends \cite[Theorem~5.21]{DP16}.
\begin{corollary}[linear convergence to a common point]
\label{c:commonpoint}
Let $w\in C:=\bigcap_{i\in I}C_i$. 
Suppose that $\{C_i\}_{i\in I}$ is superregular at $w$ and linearly regular around $w$, and that for every $j\in J$, $\{C_{s_j}, C_{t_j}\}$ is linearly regular around $w$ if $\min\{\lambda_j, \mu_j\}< 2$ and strongly regular at $w$ otherwise.
Then the cyclic gDR algorithm converges $R$-linearly locally to a point $\ox\in C$.
\end{corollary}
\begin{proof} Since strong regularity implies affine-hull regularity, the conclusion follows from Theorem~\ref{t:super} and Lemma~\ref{l:ZjCi}.
\end{proof}

We say that the cyclic gDR algorithm is \emph{fully connected} if there exist positive integers $r, q$ such that $1\leq r\leq q\leq \ell$ and (not necessarily distinct) indices $i_1, \dots, i_q\in I$ such that $\{i_1, \dots, i_q\}= I$ and that 
\begin{equation}
\label{e:full}
\{(i_1,i_2), (i_2,i_3), \dots, (i_{r-1},i_r), (i_r, i_1)\}\cup \{(i_1,i_{r+1}), \dots, (i_1,i_q)\}\subseteq \{(s_j,t_j)\}_{j\in J}.
\end{equation}
Here we adopt the following convention. If $r= 1$, then \eqref{e:full} reads as 
\begin{equation}
\label{e:anchored}
\{(i_1,i_2), \dots, (i_1,i_q)\}\subseteq \{(s_j,t_j)\}_{j\in J}, 
\text{~i.e.,~} \{(i_1, k)\}_{k\in I\smallsetminus \{i_1\}}\subseteq \{(s_j, t_j)\}_{j\in J},
\end{equation}
which is a generalization of the cyclically anchored DR algorithm.
If $r= q$, then \eqref{e:full} reads as 
\begin{equation}
\label{e:cyclic}
\{(i_1,i_2), (i_2,i_3), \dots, (i_{r-1},i_r), (i_r, i_1)\}\subseteq \{(s_j,t_j)\}_{j\in J},
\end{equation}
which is a generalization of the cyclic DR algorithm.

\begin{lemma}[shadows of common fixed points under convexity]
\label{l:cvxfixedpoints}
Suppose that $C_i$ is convex for every $i \in I$.
Let $T_{C_i,C_j}$ denote a gDR operator for the pair $(C_i,C_j)$. Let $i_1, \dots, i_r\in I$ (not necessarily distinct).
Then the following hold:
\begin{enumerate}
\item\label{l:cvxfixedpoints_anchored} 
If $\ox\in \bigcap_{k=2}^r \Fix T_{C_{i_1},C_{i_k}}$, then $P_{C_{i_1}}\ox\in \bigcap_{k=1}^r C_{i_k}$.
\item\label{l:cvxfixedpoints_cyclic}
If $\ox\in \bigcap_{k=1}^r \Fix T_{C_{i_k},C_{i_{k+1}}}$, where $i_{r+1}:= i_1$, then $P_{C_{i_1}}\ox= \dots= P_{C_{i_r}}\ox\in \bigcap_{k=1}^r C_{i_k}$.
\item\label{l:cvxfixedpoints_connected}
If the cyclic gDR algorithm is fully connected and $\ox\in \bigcap_{j\in J} \Fix T_j$, then there exists $k\in I$ such that $P_{C_{k}}\ox\in \bigcap_{i\in I} C_i$. 
\end{enumerate}
\end{lemma}
\begin{proof}
\ref{l:cvxfixedpoints_anchored}: For every $k\in \{2, \dots, r\}$, since $C_{i_1}\cap C_{i_k}\neq \varnothing$, Lemma~\ref{l:cvxFix}\ref{l:cvxFix_inconsistent} implies that $P_{C_{i_1}}\ox\in C_{i_1}\cap C_{i_k}$. Hence, $P_{C_{i_1}}\ox\in \bigcap_{k=1}^r C_{i_k}$. 

\ref{l:cvxfixedpoints_cyclic}: This is basically similar to the argument of \cite[Theorem~3.1]{BT14}. For every $k\in \{1, \dots, r\}$, by Lemma~\ref{l:cvxFix}\ref{l:cvxFix_inconsistent}, $P_{C_{i_k}}\ox\in C_{i_k}\cap C_{i_{k+1}}\subseteq C_{i_{k+1}}$, and then, by \cite[Theorem~3.16]{BC17}, $\scal{x-P_{C_{i_{k+1}}}}{P_{C_{i_k}}-P_{C_{i_{k+1}}}}\leq 0$. It follows that
\begin{subequations}
\begin{align}
0\leq \sum_{k=1}^r \|P_{C_{i_k}}-P_{C_{i_{k+1}}}\|^2&= 2\sum_{k=1}^r \scal{-P_{C_{i_{k+1}}}}{P_{C_{i_k}}-P_{C_{i_{k+1}}}} \\
&= 2\sum_{k=1}^r \scal{x-P_{C_{i_{k+1}}}}{P_{C_{i_k}}-P_{C_{i_{k+1}}}}\leq 0,
\end{align}
\end{subequations}
which yields $P_{C_{i_k}}= P_{C_{i_{k+1}}}$ for all $k\in \{1, \dots, r\}$. 

\ref{l:cvxfixedpoints_connected}: Combine \ref{l:cvxfixedpoints_anchored} and \ref{l:cvxfixedpoints_cyclic}.
\end{proof}

As one would hope for, the linear convergence of the cyclic gDR algorithm is global in the convex case. The next result encompasses \cite[Corollary~8.2 and Theorem~8.5]{BNP15}. 
\begin{theorem}[global linear convergence under convexity]
\label{t:cvxcDR}
Suppose that for every $i \in I$, $C_i$ is convex and that $\bigcap_{i\in I_p} C_i\cap \bigcap_{i\in I\smallsetminus I_p} \reli C_i\neq \varnothing$, where $I_p$ is the set of all $i\in I$ such that $C_i$ and $C_k$ are polyhedral whenever $(i, k)\in \{(s_j, t_j), (t_j, s_j)\}_{j\in J}$. Set  
\begin{equation}
\label{e:Yj}
\forall j\in J,\quad Y_j :=\begin{cases}
C_{s_j}\cap C_{t_j} &\text{~if~} \min\{\lambda_j, \mu_j\}< 2, \\
(C_{s_j}\cap C_{t_j})+ (L_j-L_j)^\perp &\text{~if~} \lambda_j =\mu_j =2 \text{~and~} \reli C_{s_j}\cap \reli C_{t_j}\neq \varnothing, \\
\Fix T_j &\text{~otherwise}.
\end{cases}
\end{equation}
Then the following hold:
\begin{enumerate}
\item\label{c:cDR_gen} 
Regardless of the starting point $x_0$, the cyclic gDR sequence $(x_n)_\nnn$ generated by $(T_j)_{j\in J}$ converges $R$-linearly to a point 
\begin{equation}
\ox\in \bigcap_{j\in J} Y_j\subseteq \bigcap_{j\in J} \Fix T_j, 
\end{equation}  
while the ``shadow sequence'' $(P_{C_i}x_n)_\nnn$ also converges $R$-linearly to $P_{C_i}\ox$ for every $i\in I$.
\item\label{c:cDR_int} 
If $0\in \inte(C_{t_j} -C_{s_j})$ whenever $\lambda_j =\mu_j =2$, then the limit point $\ox$ in \ref{c:cDR_gen} even satisfies $\ox\in C$.
\item\label{c:cDR_connect} 
If the cyclic gDR algorithm is connected and $\reli C_{s_j}\cap \reli C_{t_j}\neq \varnothing$ whenever $\lambda_j =\mu_j =2$, then the limit point $\ox$ in \ref{c:cDR_gen} satisfies $P_{C_1}\ox =\cdots =P_{C_m}\ox\in C$. 
\item\label{c:cDR_fconnect} 
If the cyclic gDR algorithm is fully connected, then the limit point $\ox$ in \ref{c:cDR_gen} satisfies $P_{C_k}\ox\in C$ for some $k\in I$. 
\end{enumerate}
\end{theorem}
\begin{proof}
First, it follows from Lemma~\ref{l:Fix}\ref{l:Fix_inclusion} that 
\begin{equation}
\label{e:YjFixTj}
\forall j\in J,\quad C_{s_j}\cap C_{t_j}\subseteq Y_j\subseteq\Fix T_j.
\end{equation}

\ref{c:cDR_gen}: For every $j\in J$, by the convexity of $C_{s_j}$ and $C_{t_j}$ and by noting from Lemma~\ref{l:cvxFix}\ref{l:cvxFix_inconsistent} that $\Fix T_j=C_{s_j}\cap C_{t_j}+ N_{C_{s_j}-C_{t_j}}(0)$ whenever $\lambda_j= \mu_j= 2$, we have that $Y_j$ is convex. 
Set $J_p:= \menge{j\in J}{C_{s_j} \text{~and~} C_{t_j} \text{~are polyhedral}}$. Then $Y_j$ is polyhedral for every $j\in J_p$. 

Let $j\in J\smallsetminus J_p$. Since $j\notin J_p$, we must have $s_j, t_j\notin I_p$ and, by assumption, $\reli C_{s_j}\cap \reli C_{t_j}\neq \varnothing$, so $Y_j= C_{s_j}\cap C_{t_j}$ if $\min\{\lambda_j, \mu_j\}< 2$, and $Y_j= (C_{s_j}\cap C_{t_j})+ (L_j-L_j)^\perp$ otherwise. Using \cite[Corollary~6.6.2]{Roc70},
\begin{equation}
\reli Y_j \supseteq\begin{cases}
\reli(C_{s_j}\cap C_{t_j}) &\text{~if~} \min\{\lambda_j, \mu_j\}< 2, \\
\reli(C_{s_j}\cap C_{t_j}) +\reli(L_j-L_j)^\perp &\text{~otherwise},
\end{cases}
\end{equation}
thus $\reli Y_j\supseteq \reli(C_{s_j}\cap C_{t_j}) =\reli C_{s_j}\cap \reli C_{t_j}$ due to \cite[Theorem~6.5]{Roc70}. By combining with \eqref{e:parameters_I&J} and noting that if $i\in I_p$, then $i\in \{s_j,t_j\}$ for some $j\in J_p$,  
\begin{equation}
\label{e:Yj's}
\bigcap_{j\in J_p} Y_j\cap \bigcap_{j\in J\smallsetminus J_p} \reli Y_j
\supseteq \bigcap_{j\in J_p} (C_{s_j}\cap C_{t_j})\cap \bigcap_{j\in J\smallsetminus J_p} (\reli C_{s_j}\cap \reli C_{t_j})
\supseteq \bigcap_{i\in I_p} C_i\cap \bigcap_{i\in I\smallsetminus I_p} \reli C_i\neq \varnothing.
\end{equation}
Let $x_0\in X$ and let $(x_n)_\nnn$ be the cyclic sequence generated by $(T_j)_{j\in J}$ with starting point $x_0$. Choose $\delta\in \RPP$ and $w\in C$ such that $\delta\geq 2\|x_0 -w\|\geq 2 d_C(x_0)$. Then $x_0\in \ball{w}{\delta/2}$. 
From \eqref{e:Yj's} and \cite[Corollary~5]{BBL99}, there is $\kappa\in \RPP$ such that $\{Y_j\}_{j\in J}$ is $\kappa$-linearly regular on $\ball{w}{\delta/2}$.
Let $j\in J$. By assumption, either $\reli C_{s_j}\cap \reli C_{t_j}\neq \varnothing$ or $C_{s_j}$ and $C_{t_j}$ are polyhedral with $C_{s_j}\cap C_{t_j}\neq \varnothing$, so we derive from Proposition~\ref{p:gcoer-DR} that $T_j$ is $(Y_j, \nu_j)$-quasi coercive on $\ball{w}{\delta/2}$ for some $\nu_j\in \RPP$.
By convexity and Lemma~\ref{l:cont-avg}\ref{l:cont-avg_avg}, $T_j$ is $\alpha_j/(1+\hbeta_j)$-averaged, which implies that $T_j$ is $(\Fix T_j, 1, \frac{1-\alpha_j+\hbeta_j}{\alpha_j})$- and hence $(Y_j, 1, \frac{1-\alpha_j+\hbeta_j}{\alpha_j})$-quasi firmly Fej\'er monotone on $X$. 
Now Theorem~\ref{t:lincvg}\ref{t:lincvg_global} yields the $R$-linear convergence of the sequence $(x_n)_\nnn$ to a point $\ox\in \bigcap_{j\in J} Y_j$. 

Next, for every $i\in I$, by Fact~\ref{f:cvxproj}\ref{f:cvxproj_avg}, $P_{C_i}$ is nonexpansive, so $\|P_{C_i}x_n -P_{C_i}\ox\|\leq \|x_n -\ox\|$ for all $n$, and we deduce that $P_{C_i}x_n$ converges $R$-linearly to $P_{C_i}\ox$.

\ref{c:cDR_int}: In the case where $\lambda_j= \mu_j= 2$, since $0\in \inte(C_{t_j}-C_{s_j})$, Lemma~\ref{l:cvxFix}\ref{l:cvxFix_intersect} yields $\Fix T_j= C_{s_j}\cap C_{t_j}$, and then, by \eqref{e:YjFixTj}, $Y_j= C_{s_j}\cap C_{t_j}$. We deduce that $Y_j= C_{s_j}\cap C_{t_j}$ for every $j\in J$, which together with \eqref{e:parameters_I&J} yields $\bigcap_{j\in J} Y_j=\bigcap_{i\in I} C_i= C$.

\ref{c:cDR_connect}: Combine \ref{c:cDR_gen} and Lemma~\ref{l:cfixedpoints}.

\ref{c:cDR_fconnect}: This follows from \ref{c:cDR_gen} and Lemma~\ref{l:cvxfixedpoints}\ref{l:cvxfixedpoints_connected}.
\end{proof}

When specialized to the case of two sets, our results cover Theorems~4.3, 4.7, and~4.14 in \cite{Pha14} where $R$-linear convergence is proved for the classical DR algorithm.
\begin{corollary}[linear convergence of gDR algorithm]
\label{c:gDR}
Let $A$ and $B$ be closed subsets of $X$, $L:=\aff(A\cup B)$, and $w\in A\cap B\neq\varnothing$. Let $\lambda, \mu \in \left]0, 2\right]$, $\hbeta :=\big(\frac{\lambda}{2-\lambda} +\frac{\mu}{2-\mu}\big)^{-1}$, $\alpha \in \big]0, 1+\hbeta\big[$, and $T$ the gDR operator for $(A, B)$ with parameters $(\lambda, \mu, \alpha)$. Then the gDR algorithm generated by $T$
\begin{enumerate}
\item\label{c:gDR_gen}
locally converges with $R$-linear rate to a point $\ox$ in each of the following cases:
	\begin{enumerate}
	\item 
	 $\{A, B\}$ is superregular and affine-hull regular at $w$, in which case $\ox\in A\cap B +(L-L)^\perp$ with $P_A\ox= P_B\ox\in A\cap B$, and if additionally $\{A, B\}$ is strongly regular at $w$, then $\ox\in A\cap B$.
	\item 
	$\min\{\lambda, \mu\}< 2$ and $\{A, B\}$ is superregular at $w$ and linearly regular around $w$, in which case $\ox\in A\cap B$.   
	\end{enumerate}
\item\label{c:gDR_cvx}
globally converges with $R$-linear rate to a point $\ox$ in each of the following cases:
	\begin{enumerate}
	\item\label{c:gDR_cvx_reli} 
	 $A$ and $B$ are convex and $\reli A\cap \reli B\neq \varnothing$, in which case $\ox\in A\cap B +(L-L)^\perp$ with $P_A\ox= P_B\ox\in A\cap B$, and if additionally $0\in \inte(B-A)$, then $\ox\in A\cap B$.
	\item\label{c:gDR_cvx_poly}
	$A$ and $B$ are polyhedral, in which case $\ox\in \Fix T\subseteq A\cap B+ N_{A-B}(0)$ with $P_A\ox\in A\cap B$. 
	\item\label{c:gDR_cvx_lin} 
	$\min\{\lambda, \mu\}< 2$, $A$ and $B$ are convex and $\{A, B\}$ is boundedly linearly regular, in which case $\ox\in A\cap B$.   
	\end{enumerate} 
\end{enumerate}
\end{corollary} 
\begin{proof} 
Applying Theorem~\ref{t:super} (noting that affine-hull regularity implies linear regularity due to Proposition~\ref{p:Lstr}) and Theorem~\ref{t:cvxcDR} with $m =2$, $\ell =1$, and $(s_1,t_1)= (1,2)$, we get \ref{c:gDR_gen} and \ref{c:gDR_cvx_reli}--\ref{c:gDR_cvx_poly}. 

We now prove \ref{c:gDR_cvx_lin}. Let $x_0\in X$ and let $(x_n)_\nnn$ be the cyclic sequence generated by $(T_j)_{j\in J}$ with starting point $x_0$. Take $\delta\in \RPP$ and $w\in A\cap B$ such that $\delta\geq 2\|x_0 -w\|\geq 2 d_{A\cap B}(x_0)$. Then $x_0\in \ball{w}{\delta/2}$. 
Since $A$ and $B$ are convex, Lemma~\ref{l:cont-avg}\ref{l:cont-avg_avg} implies that $T$ is $\alpha/(1+\hbeta)$-averaged, hence it is $(\Fix T, 1, \frac{1-\alpha+\hbeta}{\alpha})$- and also $(A\cap B, 1, \frac{1-\alpha+\hbeta}{\alpha})$-quasi firmly Fej\'er monotone on $X$. 
According to Proposition~\ref{p:gcoer-DR-lin}, $T$ is $(A\cap B, \nu)$-quasi coercive on $\ball{w}{\delta/2}$ for some $\nu\in \RPP$. Now apply Theorem~\ref{t:lincvg}\ref{t:lincvg_global} to $T$ and the corresponding set $A\cap B$.
\end{proof}

\section{Conclusion}
\label{s:con}

In this paper, we have presented a diverse collection of improvements on the linear convergence of the (cyclic) generalized Douglas--Rachford algorithm for solving feasibility problems. Our results indicate that one has great flexibility in choosing suitable parameters and still achieve convergence with linear rate.
For instance, we have proved that the generalized DR algorithm involving at most one reflection converges $R$-linearly locally assuming only that the system of superregular sets $\{A,B\}$ is linearly regular around the reference point. Interestingly, it remains open even in the convex case whether the classical DR algorithm (i.e., $\lambda=\mu=2$ and $\alpha=1/2$) converges with $R$-linear rate under the same assumption.

\section*{Acknowledgments}
The authors are thankful to the editors and the referees for their constructive comments. MND was partially supported by the Australian Research Council (ARC) under Discovery Project 160101537 and by Vietnam National Foundation for Science and Technology Development (NAFOSTED) under grant 101.02-2016.11.
This work was essentially completed during MND's visit to UMass Lowell in September 2017 to whom he acknowledges the hospitality.

\end{document}